\documentclass[num-refs]{wiley-article}

\usepackage{ifthen}
\usepackage{tikz}
\usetikzlibrary{calc,arrows,automata}
\usetikzlibrary{decorations.pathreplacing}

\usepackage{algorithm}
\usepackage{algpseudocode}
\usepackage{amsmath}
\usepackage{algorithm}
\usepackage{algpseudocode}
\usepackage{amsmath}
\usepackage{multirow}
\usepackage{colortbl}
\usepackage{caption}
\usepackage[scr=boondoxo,scrscaled=1.05]{mathalfa}
\usepackage{subcaption}

\algnewcommand\algorithmicforeach{\textbf{for each}}
\algdef{S}[FOR]{ForEach}[1]{\algorithmicforeach\ #1\ \algorithmicdo}
\newtheorem{defn}{Definition}
\newtheorem{prop}{Proposition}
\newtheorem{question}{Question}
\newtheorem{conjecture}{Conjecture}
\papertype{Original Article}
\paperfield{arXiv} 

\title{Counterexample to a Boesch's Conjecture}

\author[1\authfn{1}]{Nicole Rosenstock }
\author[2\authfn{2}]{Eduardo A. Canale }

\affil[1]{Unidad de Bioestad\'istica, Departamento de Salud P\'ublica, Facultad de Veterinaria, Universidad de la Rep\'ublica, Montevideo, 13000, Uruguay}
\affil[2]{Instituto de Matem\'atica y Estad\'istica, 
Facultad de Ingenier\'ia, Universidad de la Rep\'ublica, Montevideo, 11400, Uruguay}

\corraddress{Nicole Rosenstock,   Universidad de la Rep\'ublica, Montevideo, 11400, Uruguay}
\corremail{niky04@gmail.com}

\fundinginfo{Comisi\'on Sectorial de Investigaci\'on Cient\'ifica (CSIC)}

\runningauthor{Rosenstock et al.}

\begin{document}

\maketitle

\begin{abstract}
A key issue in network reliability analysis.
A graph with $n$ nodes and  whose $e$ edges fail independently with probability $p$ is an \emph{Uniformly Most Reliable Graph} (UMRG) if it has the highest reliability among all graphs with the same order and size for every value of $p$. The \emph{all-terminal reliability} is a polynomial in $p$ which defines the probability of a network to remain connected if some of its components fail. If the coefficients of the reliability polynomial are maximized by a graph, that graph is called  \textit{Strong Uniformly Most Reliable Graph}  (SUMRG) and it should be UMRG.
An exhaustive computer search of the SUMRG with vertices up to 9 is done. Regular graphs with 10 to 14 vertices that maximize tree number are proposed as candidates to UMRG.
As an outstanding result a UMRG with 9 vertices and 18 edges which has girth 3 is found, so smaller than the conjectured by Boesch in  1986. A new conjecture about UMRG's topology is posed here: the $(n,e)$-UMRG is $\overline{(k-1)C_3\cup C_{3+r}}$ whenever $n=3k+r$,$n\geq5$ and $e={n(n-3)}/{2}$.
A reformulation of Boesch's conjecture is presented stating that if a $(n, {kn}/{2})$-UMRG exists and it has girth $g$, then it has maximum girth  among all $k$-regular $(n,{kn}/{2})$ graphs and minimum number of $g$-cycles among those $k$-regular $(n,{kn}/{2})$ graphs with girth $g$.

\end{abstract}

\section{Introduction}

	In the past few decades, network reliability analysis has been gaining increasing interest among computer scientists and mathematicians. Based on both graph-theoretical and probabilistic models, network reliability intends to determine the probability of correct operation of a system, knowing the reliability of its components~\cite{ballcombinat}. 
	
	A network is modelled as a simple, undirected graph $G=(V,E)$. Different reliability problems should emerge depending on whether nodes or edges fail. In this paper we will assume that nodes are perfect and edges fail with identical and independent probability $q = 1 - p$, $0<q<1$. Moreover, a decision should be made on which network reliability measure to choose. We will focus on the \textit{all-terminal reliability} measure. The all-terminal reliability $R_G(p)$, defined as the probability that a network remains connected after some of its components fail, is a polynomial in $p \in [0,1]$. A precise definition of the reliability polynomial, as presented in~\cite{Colbourn91}, is the following:
	let $N_i$ denote the number of connected spanning subgraphs with $i$ edges. Then the reliability is a polynomial in p,
	\begin{equation}\label{reliab}
		R_G(p)=\sum_{i=0}^{e}N_ip^i(1-p)^{e-i}
	\end{equation}

	where $e$ is the total number of edges in $G$.

	For convenience, we will be dealing, not only with the reliability polynomial but also with unreliability polynomial $U_G(p)$~\cite{boesch86}, defined as follows: let $m_k$ be the number of edge disconnecting sets of size k. Then the probability that $G$ becomes disconnected can be expressed as
	\begin{equation}\label{unreliab}
		U_G(p)=\sum_{k=0}^{e}m_k(1-p)^kp^{e-k},
	\end{equation}	
	where $U_G=1-R_G$.

	Given a graph with $n$ nodes and $e$ edges, we call any graph with the same order and size a $(n,e)$-graph. For any $n$ and $e$, if $p$ is fixed, it is clear that there always exists a graph that minimizes $U_G(p)$, i.e $U_G(p) < U_H(p)$, for all $(n,e)$-graphs $H$. Analogously, with $p$ fixed, there always exists a graph that maximizes $R_G(p)$, i.e $R_G(p) < R_H(p)$, for all $(n,e)$-graphs $H$. When $p$ is fixed and it is either large or small, the problem of finding a graph that maximizes reliability has been reduced to a graph theoretic problem. Some of these findings are thoroughly described in~\cite{boesch86}. For small $p$, the graph that minimizes unreliability is a $\lambda$-optimal graph. Namely, $m_\lambda(G)$ is minimum among all $(n,e)$-graphs with maximum edge-connectivity $\lambda$. On the contrary, for large $p$,  the graph that minimizes unreliability is a $t$-optimal graph. Namely, a graph with maximum number of spanning trees. 	
	A graph with $n$ nodes and $e$ edges that minimizes $U_G(p)$, or maximizes $R_G(p)$, for every $p \in [0,1]$ is called \emph{Uniformly Most Reliable Graph} (UMRG).
	
	Given a graph with $e$ edges and its unreliability polynomial $U(G)$, the $e+1$ vector $(m_0,m_1,\dots,m_e$) is called the \textit{edge disconnecting vector} or also, \textit{edge-cut frequency vector}~\cite{boesch1985cut}. As stated in~\cite{boesch86b}, a sufficient condition for a graph $G$ to be UMRG is that it minimizes every term in the edge disconnecting vector, i.e $m_i(H)\geq m_i(G)$, for all $i$ and all $(n,e)$-graphs. 
	Let us call such graphs \emph{Strong Uniformly Most Reliable Graph} (SUMRG).
	
	By definition, there cannot exist disconnecting sets whose cardinality is lower than the edge connectivity $\lambda$. Therefore, $m_i(G)=0$, for all $i<\lambda$ while $\lambda$ must be maximum~\cite{bauer1987validity}. In terms of connected spanning subgraphs, this is equivalent to say that 
	\begin{equation}
	N_{e-i} = \left(\genfrac{}{}{0pt}{}{e}{e-i}   \right) \qquad \forall i<\lambda.
	\end{equation}
	Other coefficients where also described by Van Slyke and Frank~\cite{Slyke71} as follows:
	
	\begin{align}
		N_{e-i} &= 0,\quad i>e-n+1, \label{cero}\\
		N_{e-i} &= n_t, \quad i =e-n+1, \label{arboles}\\
		N_{\lambda} &= \left(\genfrac{}{}{0pt}{}{e}{\lambda}   \right) - n_\lambda, \label{mincut}
	\end{align}
	where $n_t$ is the number of spanning subtrees and $n_\lambda$ the number of minimum cardinality network cuts. It is worth noting that for each $i$, $m_i =  \left(\genfrac{}{}{0pt}{}{e}{i}   \right) - N_{e-i}$. Since calculation of the polynomial coefficients is 
	NP-hard as proved by ~\cite{provanball}, the coefficients described above can reduce significantly the computational effort required to determine the reliability polynomial for any $(n,e)$-graph. The problem of finding the edge connectivity of a graph was solved by Ford and Fulkerson~\cite{fordful}. 
	On the other hand, according to the Kirchhoff Matrix Tree Theorem, the number of spanning subtrees equals any cofactor of the Laplacian matrix~\cite{kirchhoff1847ueber} which can be easily computed in polynomial time. Finally, Ball and Provan~\cite{ballprovanAlg} describe an algorithm in order to compute $n_\lambda$, the number of minimum cardinality network cuts, in polynomial time. Therefore, most of the computational effort should be directed towards the calculation of the remaining coefficients. Once the coefficients of $R_G$ (or $U_G$) are computed for every $(n,e)$-graph, finding the SUMRG implies choosing the graph that maximizes all coefficients $N_i$ (or minimizes every $m_i$), as long as it exists one that satisfies this condition.
	
	In the context of reliable network synthesis, the following corollary arises as a consequence of Boesch's Theorem linking spanning graphs to graph reliability~\cite{boesch09}: if $G$ is uniformly most-reliable, then:
	\begin{enumerate}
		\item $G$ has the maximum number of spanning trees among all simple $(n,e)$-graphs. Namely, UMRG are $t$-optimal and
		\item $G$ is max-$\lambda$, i.e., has the maximum possible value of $\lambda$ among all simple $(n,e)$-graphs, where $\lambda(G) = \lfloor 2e/n \rfloor$, and the minimum number $m_\lambda$ among all max-$\lambda$ graphs.
	\end{enumerate}
	
	Apart from these properties, several conjectures have been proposed by authors: a uniformly-most reliable network has degrees almost regular, maximum girth, minimum diameter, no multiple edges and that regular complete multipartite graphs are uniformly-most reliable~\cite{myrvold1996reliable}. Ath \& Sobel~\cite{ath_sobel_2000} proposed possible  counterexamples for two of these conjectures: the diameter and the girth conjecture. Nevertheless, they did not prove whether the graphs were UMRG.
   Let us highlight the last mentioned conjecture.
   \begin{conjecture}\label{conjecture}
   The  UMRG has maximum girth among the graphs with the same order and size. (Boesch\cite{boesch86})
   \end{conjecture}
   
	Boesch had once conjectured that UMRG always exist~\cite{boesch86}, ``not realizing that Kelman's \cite{kelmans1981graphs} had much earlier (1981) published an infinite family of counterexamples''\cite{myrvold1996reliable}. Also, this author found UMRG of 8 or fewer nodes by finding SUMRG. Although these graphs are presented in~\cite{myrvoldTech}, as far as we know, this technical report is not available in the literature.

	Studies over the past decades have provided evidence of the existence of UMRG in certain classes of regular graphs. Among $(n,n)$-graphs, cycle graph $C_n$ is UMRRG~\cite{grossUOR}.  Graph $K_n$, the only graph with $n$ nodes and $\binom{n}{2}$ edges, is therefore UMRRG as well. Additionally, Kelmans showed that when $\binom{n}{2} - \lfloor n/2\rfloor \leq e \leq \binom{n}{2} $ there exists a UMRG and it corresponds to the complete graph $K_n$ with a matching removed~\cite{kelmans1981graphs}, therefore, $K_n$ minus a perfect matching is a  UMRRG. Other regular graphs were also found to be UMRG such as $K_{3,3}$~\cite{wangn3}, Wagner~\cite{wagner}, Petersen~\cite{petersen}, Yutsis~\cite{yutsis}, $K_{4,4}$~\cite{grafok44}, $\overline{K_4\cup C_3}$~\cite{myrvoldTech}. Both Heawood~\cite{ath_sobel_2000b} and Möbius-Kantor graphs~\cite{bourel} are potential UMRRG, however, there do not exist formal proofs so far. Similarly, there do not exist proofs for the still-open conjecture that regular complete bipartite graphs are UMRG~\cite{boesch86b}. Despite that Cheng~\cite{cheng1981maximizing} demonstrated their $t$-optimality, only the aforementioned $K_{4,4}$ was proven to be UMRG. 
	As we mentioned above, there are no UMRG for some values  $(n, e)$, but  these values of $(n,e)$ never correspond to regular graphs. So, the following  question naturally arises
 \begin{question}\label{question}
If $kn=2e$, is there an UMRG among the $(n,e)$-graphs?
	\end{question}

	In this paper, we will restrict the study of uniformly most reliable graphs to the class of 2-connected $k$-regular graphs. The importance of biconnectivity was independently established by Boesch et al.~\cite{boesch86b} and Wang~\cite{wangn3}. While the former proved that if $G$ is $t$-optimal with $e\geq n\geq3$, then it is 2-connected, the latter proved that if $G$ is SUMRG  then it is 2-connected.
	
	The coefficients of the candidates to UMRG for $n\leq 9$,  $e \leq k\cdot n /2 $ and $k=3,\dots,n-3$ have been computed and compared against other $(n,e)$-graphs. The contribution of this work is:
	\begin{enumerate}
	\item providing a list of regular uniformly most reliable graphs, for graphs with $n\leq 9$,
	\item present a counter-example to Conjecture~\ref{conjecture},
	\item prove that a positive answer to  Question~\ref{question} is incompatible with  Conjecture~\ref{conjecture}.
	\item describe properties such as energy and Laplacian energy of graphs for the UMRG.
	\end{enumerate}

		\section{Materials and Methods}
	
	Most of the computational experiments were conducted using R language (\cite{RLang}) and the \textit{igraph} v1.3.0 R library~\cite{igraph}, on a Home-PC with Intel R
CoreTM i5-3470U, 3.20GHz, 64-bit OS. Experiments requiring larger computational effort were performed in ClusterUY (\cite{nesmachnow2019cluster}).  
Using the program Nauty v2.7~\cite{nauty}, a list of 2-connected (n,e)-graphs was generated, where $6\leq n \leq 14$ and $e = k\cdot n/2$, with $k \in \{3,\dots,n-3\}$. Nauty allows a fast generation of non-isomorphic small graphs of a specified class.  
	
\subsection{Computation of coefficients $N_i$ and $m_k$ }

	An algorithm was implemented in order to count the number of connected spanning subgraphs with $i$ edges for each graph, namely $N_i$, with $0\leq i\leq e$~(Algorithm \ref{Ncoef}). 
	
\begin{algorithm}[H]
	\caption{N = ComputeNi(G(V,E))} \label{Ncoef}
	\begin{algorithmic}[1]
		\Require{$G= (n,e)$-graph}
		\Ensure{$\mathcal{N}= (N_0,\ldots,N_e)$}
		\State $\mathcal{N} = (0,\dots,0)$
		\State $\lambda = EdgeConnectivity(G)$
			\For {$k=e:0$}  \label{forstart}
			\If{$k>e-n+1$}
			\State{$\mathcal{N}_{e-k}=0$}
			\ElsIf{$k=e-n+1$}
			\label{STrees}\State{$\mathcal{N}_{e-k}=SpanningTrees(G)$}
			\label{lambda}\ElsIf{$e-k=\lambda$}
			\State{$\mathcal{N}_{\lambda}= \left(\genfrac{}{}{0pt}{}{e}{\lambda}   \right) - MinCardinalCut(G,\lambda) $}
			\Else{
			\State {$\mathcal{C}= $ All possible sets of $k$ edges} \label{bAll}
				\ForEach{$c \in C$}
				\State {$G'  \leftarrow G - {c}$}
				\If{$G'$ is connected}
				\State{$\mathcal{N}_{e-k} = \mathcal{N}_{e-k}+1$}
				\EndIf
			\EndFor\label{eAll}
			\EndIf}
			\EndFor  \label{forend}
		\State \Return {$\mathcal{N}$}
	\end{algorithmic}
\end{algorithm}

\begin{algorithm}[H]
	\caption{$N_{\lambda}$ = MinCardinalCut(G(V,E))} \label{Nlambda}
	\begin{algorithmic}[1]
		\Require{$G= (n,e)$-graph}
		\Ensure{$N_{\lambda}$}
		\State $\lambda = EdgeConnectivity(G)$
		\State $G' = Directed(G)$
		\State $N_{\lambda}=0$
		\State Choose random node $s$ in $G'$
		\ForEach {$t \in V - \{s\}$}
		\State $c = MinCut(G',s,t)$
		\If {$c=\lambda$}
		\State {$N_{\lambda} = N_{\lambda} + stMinCuts(G',s,t)$}
		\EndIf
		\State {Collapse nodes $s$, $t$ into single node $s$}
		\EndFor
		\State \Return{$N_{\lambda}$}
	\end{algorithmic}
\end{algorithm}

Algorithm~\ref{Ncoef} receives a graph $G$ as an input. A FOR loop (Lines~\ref{forstart}-\ref{forend}) computes $N_i$ the number of spanning subgraphs with $i$ nodes. In each iteration a set of $k$ edges is removed and the coefficient is determined. To increase algorithm efficiency and lessen computation time, equations~\eqref{cero},~\eqref{arboles} and~\eqref{mincut} were used when calculating the coefficients that met the conditions detailed above. The spanning trees were determined  by calculating a cofactor of the Laplacian matrix of the graph $G$ (Line~\ref{STrees}). To compute $N_\lambda$ (Line~\ref{lambda}) an adaptation of the algorithm CUTENUM described in~\cite{ballprovanAlg} was developed: the \textit{MinCardinalCut} algorithm~(Algorithm \ref{Nlambda}). While CUTENUM was originally developed for directed graphs, \textit{MinCardinalCut} was adapted to undirected graphs. For the remaining coefficients, a brute-force approach was conducted (Line~\ref{bAll} - \ref{eAll}). A list $\mathcal{C}$ of sets of $k$ edges was generated (Line 11). For each set in $\mathcal{C}$, if the removal of the set from the graph $G$ yielded a connected spanning subgraph the coefficient $\mathcal{N}_{e-k}$ was incremented by 1 (Line 12-17). Finally, the $(e+1)$-vector $\mathcal{N}$ containing coefficients $N_{e-k}$ with $k=0,\ldots,e$ was returned.

Algorithm \textit{MinCardinalCut} calculates the number of minimum cardinality cuts in an undirected graph $G$. In line 1, edge connectivity is calculated for graph $G(V,E)$. Graph $G(V,E)$ is converted into a directed graph $G'(N,E')$: every edge $(i,j) \in E$ is replaced by arcs $(i,j)$ and $(j,i) \in E'$. A node $s\in V(G)$ is chosen arbitrarily (Line 4). For each of the remaining nodes $t \in V$, the minimum st-cut $c$ is calculated using \textit{MinCut} function (Line 6). If $c=\lambda$, the amount of minimum st-cuts, calculated with function \textit{stMinCuts} is added to $N_{\lambda}$ (Line 7-9). After each iteration, nodes $s$ and $t$ are collapsed into a single node $s$ (Line 10). Functions \textit{EdgeConnectivity}, \textit{min\_cut} and \textit{stMinCuts} were provided by \textit{igraph} package~(\cite{igraph}) for R~\cite{RLang}.

For computational reasons, algorithm~\ref{Ncoef} was applied to $(n,{kn}/{2})-$graphs meeting the following conditions:
\begin{enumerate}
	\item $6\leq n \leq 7$ ; $k = 3,\ldots,n-3$,
	\item $n = 8$; $k = 3,4,5$.
\end{enumerate}
 A set of coefficients $N_i$ was found for every graph with these characteristics. 

For graphs with order $9$ another approach was taken to reduce computational effort. Since the number of spanning trees can be found in polynomial time through the Kirchoff Matrix Tree Theorem, the Laplacian matrix was computed for each graph with given sizes and orders. According to this theorem, any cofactor of the Laplacian matrix of a graph equals the number of spanning trees of the same graph. Consequently, the number of spanning trees $N_t$ was computed by finding a cofactor of the Laplacian matrix for every graph in each class of $(n,e)-$graphs. From Boesch's Theorem we know that UMRG must be $t-$optimal, hence we determine the graph with the largest $N_t$ among all the graphs in each class. The graph with the largest number of spanning trees for a given $n$ and $e$ is set as a possible candidate to be SUMRG. After choosing a candidate in each class, Algorithm~\ref{Ncoef} was used to compute $N_i$ coefficients for every candidate which were later used to find the number of edge disconnecting sets of size $k$, namely $m_k$, knowing that $m_k = C_{e-i}^e - N_i$ (Algorithm~\ref{NiMk}).

\begin{algorithm}[H]
	\caption{$\mathscr{m} = DisconnectingVector$($G$)} \label{NiMk}
	\begin{algorithmic}[1]
		\Require{$G(n,e)$}
		\Ensure{$\mathscr{m} = (m_0,\cdots,m_e)$}
		\State $\mathcal{N} = ComputeNi(G)$
		\State{$\mathscr{m}=(0,\dots,0)$}
		\ForEach{$i \in \{0,\dots,e\}$}
		\State{$\mathscr{m}_i = C_{e-i}^e - \mathcal{N}_i$}
		\EndFor 
		\State \Return {$\mathscr{m}$}
	\end{algorithmic}
\end{algorithm}

Recalling that a SUMRG minimizes every coefficient $m_k$ in its edge disconnecting vector, a new algorithm was implemented to compare the candidate to SUMRG's edge disconnecting vector $\mathscr{m}$, against the vectors of the remaining graphs in each class, for $0\leq k \leq e$. This is described in Algorithm~\ref{Mkcoef}. Beginning with an empty set $\mathcal{S}$ (Line \ref{l1}), a FOR loop (Lines~\ref{forbeg}-\ref{forend2}) generates $\mathcal{C}$, all possible sets of $k$ edges (Line~\ref{kedges}) of a given graph $G$. A WHILE loop (Lines \ref{whilebeg}-\ref{whileend}) iterates over $\mathcal{C}$. In each iteration, one set $c \in C$ is removed from $G$ and checked if the resulting spanning subgraph is disconnected. If it is, a counter $m$ is incremented by 1. The WHILE loop ends either when $m$ exceeds $\mathscr{m}_k$, the number of edge disconnecting sets of size $k$, of the candidate graph or when every $c \in \mathcal{C}$ has been checked. If the WHILE loop exited because the number of $k$-edge-disconnecting set of $G$ was greater than the candidate's $m_k$, the coordinate $\mathscr{m}_k$ will take the value \textit{More}; if it exited because all $c \in \mathcal{C}$ were checked, depending on whether $m=\mathscr{m}_k$ or $m<\mathscr{m}_k$, the coordinate $\mathscr{m}_k$ will take the value \textit{Equal} or \textit{Less}.

\begin{algorithm}[H]
	\caption{$\mathcal{S}= Compare\_m_{k}$($\mathscr{m}$ ; $G$)} \label{Mkcoef}
	\begin{algorithmic}[1]
		\Require{$\mathscr{m}$; $G=(n,e)$}
		\Ensure{$\mathcal{S} = (cm_0,\cdots,cm_e)$}
		\State{$\mathcal{S}=(0,\dots,0)$} \label{l1}
		\For{$k = e:0$} \label{forbeg}
		\State{$m=0$} 
		\State {$\mathcal{C}= $ All possible sets of $k$ edges} \label{kedges}
		\While{$(m \leq \mathscr{m}_k)$ AND (\textit{there are $c \in C$ to be checked})} \label{whilebeg}
		\State{$G'=G - c$}
		\If{$G'$ is disconnected}
		\State{$m = m+1$}
		\EndIf
		\State{$c\leftarrow Next(c)$}
		\EndWhile \label{whileend}
		\If{$m>\mathscr{m}_k $} \label{ifmk}
		\State{$\mathcal{S}_k = More$}
		\ElsIf{$m=\mathscr{m}_k $}
		\State{$\mathcal{S}_k = Equal$}
		\Else \State {$\mathcal{S}_k = Less$}
		\EndIf \label{ifmk2}
		\label{forend2}\EndFor
		\State \Return {$\mathcal{S}$}
	\end{algorithmic}
\end{algorithm}

In those classes containing a single 2-connected graph, i.e complete graph, coefficients were not computed since there were no other candidates to compare with.

\subsection{Graph Properties Computation}

With the aim to assessing graph properties and determine any distinctive characteristics of UMRG, the following six properties were computed and analyzed for all the graphs considered in this study: girth, diameter, energy, Laplacian energy, Fiedler number and number of Hamiltonian cycles. All computations were performed in RStudio~\cite{RStudio}.  Both girth and diameter were computed via \textit{igraph} v1.3.0 library~\cite{igraph}. Energy and Laplacian energy were calculated using their respective definitions. The energy $E$ of a graph $G$ is defined as follows~\cite{balakrishnan2004energy}: 
\begin{defn}
	Let $A(G)$ be the adjacency matrix of a graph $G$ and $\lambda_1,\dots,\lambda_n$ its eigenvalues. Then, $E(G)=\sum_{i=1}^{n}|\lambda_i|$. 
\end{defn}
The Laplacian energy $LE$ of a graph $G$ is defined as follows~\cite{gutman2006laplacian}:
\begin{defn}
	Let $G$ be a $(n,e)$-graph and its Laplacian values $\mu_1,\dots,\mu_n$, then $LE(G) = \sum_{i=1}^{n}|\mu_i-{2e}/{n}|$.
\end{defn}
The Fiedler number, also known as \emph{algebraic connectivity}, is the second smallest eigenvalue of the Laplacian matrix of a graph~\cite{de2007old}.
Finally, the number of Hamiltonian cycles computation was performed using \textit{sna} v2.6~\cite{snalib} library.

\subsection{Regular Larger Order Graphs}

Since the analysis of graphs with order greater than $9$ requires an enormous computational effort, only the number of spanning trees $N_t$ was computed solely for regular graphs with $n \geq 10$. Since UMRG are $t$-optimal, the graph with the largest number of spanning trees was set as candidate to UMRG for each combination of $n$ and $e$. Also, the properties mentioned above were calculated for every regular graph.

\section{Results}

\subsection{Number of connected spanning subgraphs}  

The amount of regular and non regular $(n,e)$-graphs obtained from Nauty for each $n$, with $6\leq n \leq 14$, and $e$ are detailed in Table~\ref{tablagrafos1} of Appendix~\ref{apen1}.

As described above, coefficients $N_i$ were calculated for every $(n,{kn}/{2})$-graph with order $6 \leq n \leq 8$ and only for graphs maximizing the number of spanning trees when $n=9$. For the latter, their edge disconnecting vector $\mathscr{m}$ was obtained and compared against vectors of the other graphs in their respective class to check their optimality. Classes containing a single 2-connected graph were not explored, i.e.  $(n,n)$-graphs (cycles) and $(n,{n\cdot(n-1)}/{2})$-graphs (complete graphs) or classes were already it is known the UMRG, i.e. $(2h, 2h(h-1))$-graphs (complement of a matching). In Table~\ref{tablacoef} and Table~\ref{tablacoef2}, UMRG coefficients $N_i$ are presented by graph order and size, for all $i\geq 5$. When $i\leq4$, $N_i$ equals 0 since there are no connected subgraphs with 4 or less edges with 6 or more nodes. These values were omitted in the tables. For a matter of space, coefficients $N_5$ to $N_{16}$ are presented in Table~\ref{tablacoef} of Appendix~\ref{apen2} and coefficients $N_{17}$ to $N_{27}$ are presented in Table~\ref{tablacoef2} of Appendix~\ref{apen2}.

A UMRG was found for every $n$ and $e$ considered. Consistently with what had been stated by Bauer~\cite{bauer} and Boesch~\cite{boeschsurvey}, for every $n$ and $k$ considered in this work, UMRGs turned out to be simple $k$-regular graphs.  Among $(6,9)$ and $(8,16)$-graphs, the complete bipartite graphs were the optimal one in both classes. Boesch~\cite{boesch86b} had conjectured that uniformly most reliable $(n,n+3)$ graphs with 6 or more nodes always exist and they are elementary subdivisions of $K_{3,3}$. Wang~\cite{wangn3} proved it correct. Consistently, the UMRG $(6,9)$-graph found in this work is the $K_{3,3}$. Analogously, the $K_{4,4}$ is the UMRG among graphs with 8 nodes and 16 edges. The latter was conjectured by Boesch~\cite{boesch09} and later proved by Canale et al.~\cite{grafok44}.  Among graphs with 8 nodes and 12 edges, the Möbius graphs $M_4$, also called Wagner graph, is the optimal graph as it had been already proved in~\cite{wagner}. For the remaining values of $n$ and $e$, i.e $(7,14)$,$(8,20)$,$(8,24)$ and $(9,27)$, UMRG are respectively
$ \overline{C_3\cup C_4}, \overline{C_3\cup C_5}$, and $\overline{3C_3}$ while $(9,18)$ is 
depicted in Figure~\ref{n9k4}.

\begin{figure}
\centering
		\includegraphics[width=0.5\textwidth]{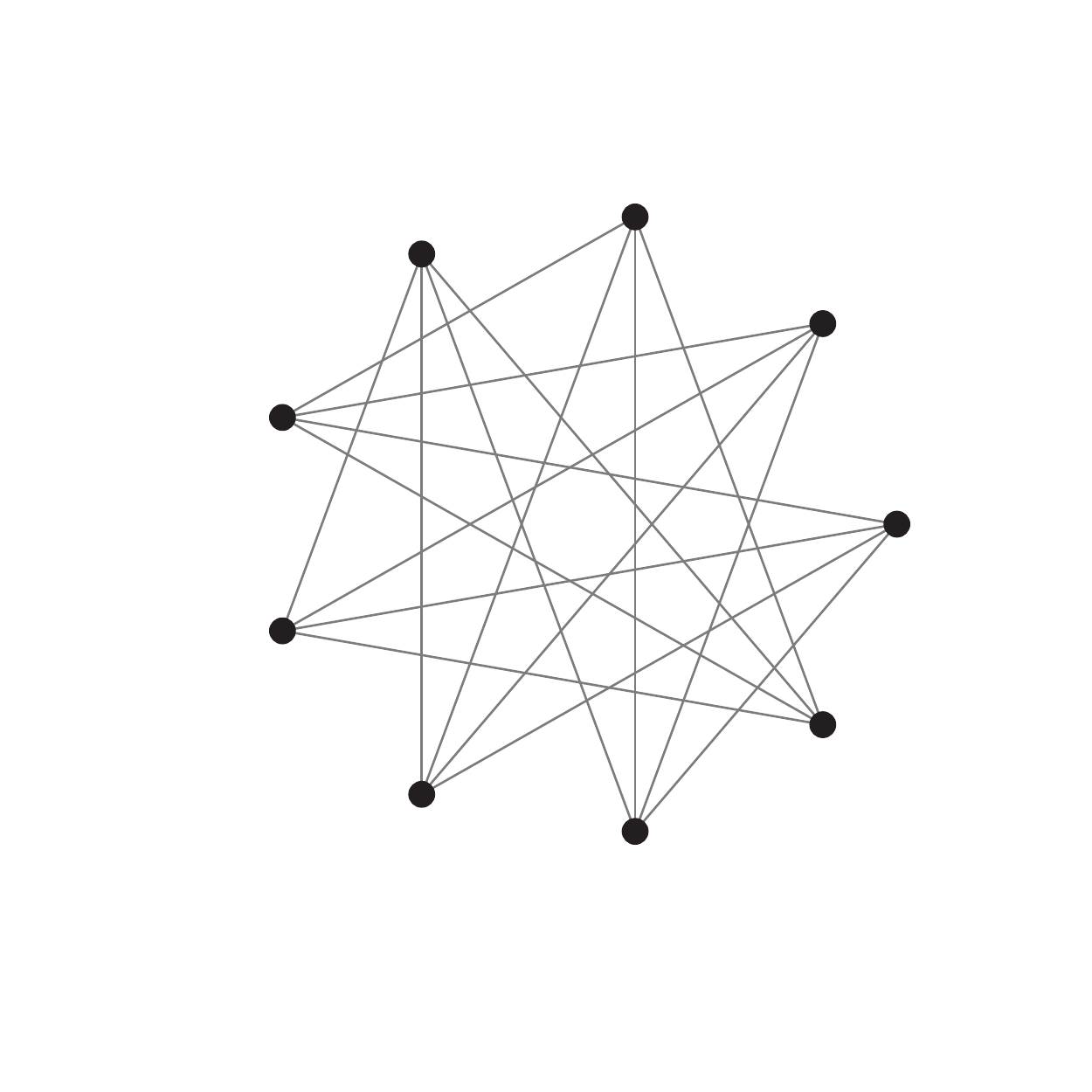}
		\caption{$n=9$, $e=18$}
		\label{n9k4}
\end{figure}

\subsection{Graph Properties}
Several properties were studied for every graph with $n\leq 9$, such as: diameter, girth, energy, Laplacian energy and Fiedler number, also known as algebraic connectivity.  

Among the graphs considered in this work, as the conjecture states, diameter was minimum for every UMRG. However, we already know that this condition is not fulfilled by other uniformly most reliable graphs as it was already proven by Ath \& Sobel~\cite{ath_sobel_2000}. Although they only compare 2 non-regular graphs in terms of their reliability polynomial, we can easily verify in modest computational time that not only is the first graph more reliable than the second one, but also it is the uniformly most reliable graph among all $(6,8)$-graphs.  

There exists another conjecture that had been around for several years stating that UMRG has maximum girth. Although Ath \& Sobel~\cite{ath_sobel_2000} suggested a potential counterexample using non-regular graphs, they did not prove that the graph with smaller girth was indeed UMRG. In this work, as a novelty, we found a UMRG with the lowest girth among all $(9,18)$-graphs, thus 4-regular. This graph was depicted in Figure~\ref{n9k4}. Among the 27015 graphs in this class, the maximum girth was 4 and it was attained by four non-regular graphs. On the other hand, the remaining graphs had a minimum girth of 3 within which we can find the UMRG. So this graph is a counterexample to Boesch conjecture that UMRG have maximum girth.  This counterexample could be one of a family of counterexamples. Indeed, consider the $k$-regular  graphs with $n_k=2k+1$ vertices and $e_k = k(2k+1)/2$ edges for even $k$. Then any of them should have girth 3 as we prove next.
\begin{prop}\label{prop}
If $k>2$ then, any  $k$-regular graph with $n_k=2k+1$ vertices and $e_k = k(2k+1)/2$ edges has girth 3. 
\end{prop}
\begin{proof}
By contradiction, suppose $G$ is a $k$-regular $n_k,e_K$ graph with girth greater than 3, and let $u_1$ and $v_1$ be two adjacent vertices. Then, $u_1$ and $v_1$ do not share any neighborhood otherwise the girth should be 3. Let  $u_2,\dots,u_k$  be the vertices adjacent with $u_1$ different from $v_1$, and $v_2,\dots,v_k$ the vertices adjacent with $v_1$ different from $u_1$, then $u_i \neq v_j$ for all $i,j$. See  left hand side graph in Figure~\ref{prop}.
By the assumption of girth greater than 3, vertices $u_2,\dots,u_k$ can not be adjacent with each other. The same happens for vertices $v_2,\dots,v_k$.

So the vertex set of $G$ consists of vertices $u_i$, $v_i$ and an extra vertex $w$. 
Suppose now that $w$ is adjacent with $h$ vertices $u_i$ and $k-h$ vertices $v_i$ as  seen in the middle of Figure~\ref{prop}.
First let us notice that  $h>0$. Indeed, if $h=0$, we will get the 3-cycle $v_1,w,v_2,v_1$ contradicting the assumption of girth greater than 3.
By symmetry we have that $h<k$
 Now, since $h>0$ let $u_i$ be adjacent with $w$. In order to avoid 3-cycles of the form $u_i,w,v_j,u_i$ we need each $u_i$ to be adjacent with vertices $v_j$ not adjacent with $w$. Therefore $k-2\leq (k-1)-(k-h)=h-1$, i.e. $k\leq h+1$ but $h+1\leq (k-1)+1=k$, so $h+1=k$, i.e. $h=k-1$. By symmetry $k-h = k-1$, i.e. $h=1$. Therefore $k-1=1$, i.e. $k=2$, contradicting the  hypothesis of $k>2$.\qed
\end{proof}

\newcommand{\grafo}{\foreach \x in {2,3,4}{
      \node (\x) at ( 0, 4-0.5*\x) [place] {};
      \node  at ( 0, 4-0.5*\x) [left] {$u_{ \x}$};
      \node (\x') at ( 4, 4-0.5*\x) [place] {};
      \node  at ( 4, 4-0.5*\x) [right] {$v_{\x}$};
    };
      \node (n) at ( 0, -0.2) [place] {};
      \node at ( 0, -0.2) [left] {$u_{k}$};
      \node (n') at ( 4, -0.2) [place] {};
      \node  at ( 4,-0.2) [right] {$v_{k}$};
      \node (u) at ( 1, 4) [place] {};
     \node (v) at ( 3, 4) [place] {};
      \node  at ( 1, 4) [left] {$u_1$};
      \node  at ( 3, 4) [right] {$v_1$};     
     \node at ( 0, 1.6)  {$\vdots$};
    \node at ( 4, 1.6)  {$\vdots$};
    
     \foreach \x in {2,3,4}{
 	  \draw (u) to (\x);
 	  \draw (v) to (\x');
 	  };
 	 \draw (u) to (n);
 	  \draw (v) to (n');
 	  
 	 \draw (u) to (v); }

\newcommand{\grafob}{
     \node (w) at ( 2, -1) [place] {};
    \node  at ( 2, -1) [below] {$w$};     
    \node(h)  at ( 0, 0.5) [place]{};     
    \node  at ( 0, 0.5) [left] {$u_{k-h+1}$};
    \node(kh)  at ( 4, 1) [place]{};     
    \node  at ( 4, 1) [right] {$v_{h+1}$};

     \foreach \x in {h,kh,n,n'}
 	  \draw (w) to (\x);

  \draw[dotted] (w)++(20:1) arc (20:35:2);
  \draw[dotted] (w)++(180-20:1) arc (180-20:180-29:2);
    
      \node at ( 0, 0.25)  {$\vdots$};
     \node at ( 4, 0.7)  {$\vdots$};
     }

\begin{figure}
    \begin{tikzpicture}
          [scale=0.8,place/.style={circle,draw=black,thick,fill=black, inner sep=0pt,minimum size=1mm}]
    \grafo
    \begin{scope}[xshift=6cm]
        \grafo
        \grafob
    \end{scope}
    \begin{scope}[xshift=12cm]
        \grafo
        \grafob
        \node(h')  at ( 4, 1.3) {};     
          \foreach \x in {2',h'}
     	  \draw (n) to (\x);
        \draw[dotted] (n)++(20:1.6) arc (20:40:1.6) node [midway,right, rotate=26]  {$k-2$};
        \node at ( 4, 0.7)  {$\vdots$};
    \end{scope}  
    \end{tikzpicture}    
\end{figure}

Notice the condition $k>2$ is needed since $C_5$ verifies the other hypothesis and has girth 5.
On the other hand, there are  non-regular $(n_k,e_k)$-graphs with girth 4. Indeed,  it is enough to describe a  bipartite $(n_k,e_k)$-graph. For instance, let $K_{k+1,k}$ be the complete bipartite graph with $k+1$ vertices in one  partition and $k$ in the other one. This graph has $2k+1$ vertices and $k(k+1)= k^2 +k$ edges, i.e. $(k^2+k)-e_k= k/2$ edges more than we need. So, it suffices to erase a matching of $k/2$ edges to obtain a connected $(n_k,e_k)$ graph  with girth 4. Figure~\ref{girth4} illustrates the idea. Therefore we have prove the following result.
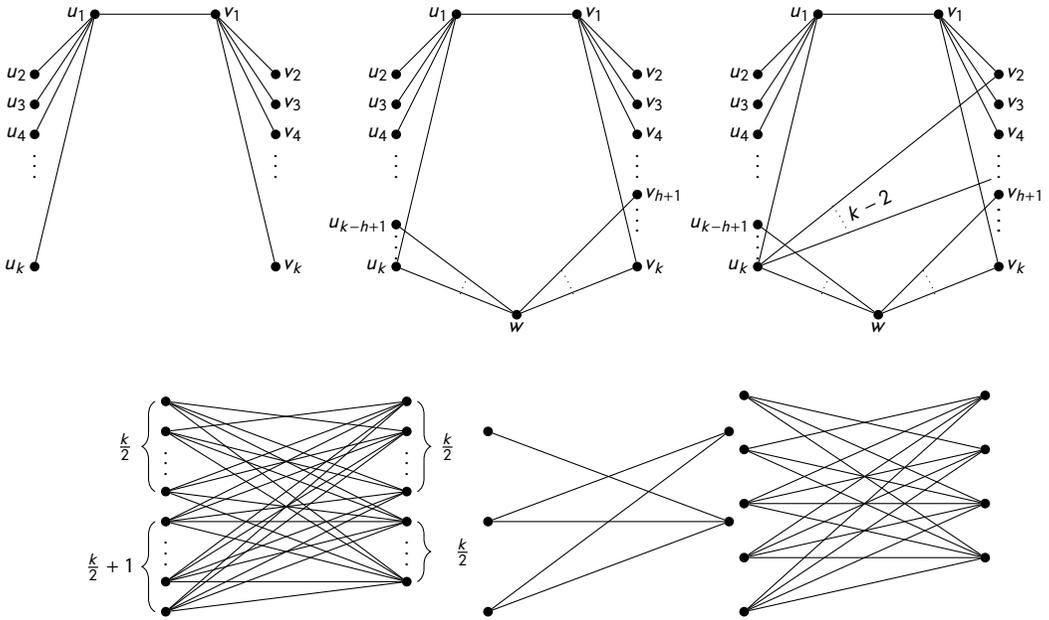
\begin{figure}
\centering

\begin{tikzpicture}
      [scale=0.8,place/.style={circle,draw=black,thick,fill=black, inner sep=0pt,minimum size=1mm}]

\foreach \x in {1,2,4,5,7,8}{
      \node (\x) at ( 0, 4-0.5*\x) [place] {};
      \ifthenelse{\x<8}{\node (\x') at ( 4, 4-0.5*\x) [place] {}}{};
    };
    \foreach \x/\y in {0/2.6,4/2.6,0/1.1,4/1.1}
     \node at ( \x, \y)  {$\vdots$};

\foreach \x in {1,2,4,5,7,8}      
\foreach \y in {1,2,4,5,7}
      \ifthenelse{{{\x<5}\AND{\NOT{\x=\y}}}\OR{\x>4}}{\draw (\x) to (\y')}{};

\draw [decorate,decoration={brace,amplitude=5pt},xshift=-5pt,yshift=0pt]
( 0, 4-0.5*4) -- ( 0, 4-0.5*1) node [black,midway,xshift=-0.4cm] 
{ $\frac{k}2$};
\draw [decorate,decoration={brace,amplitude=5pt},xshift=-5pt,yshift=0pt]
( 0, 4-0.5*8) -- ( 0, 4-0.5*5) node [black,midway,xshift=-0.6cm] 
{ $\frac{k}2+1$};
\draw [decorate,decoration={brace,amplitude=5pt},xshift=5pt,yshift=0pt]
( 4, 4-0.5*1) -- ( 4, 4-0.5*4) node [black,midway,xshift=0.4cm] 
{ $\frac{k}2$};
\draw [decorate,decoration={brace,amplitude=5pt},xshift=5pt,yshift=0pt]
( 4, 4-0.5*5)--( 4, 4-0.5*7)  node [black,midway,xshift=0.6cm] 
{ $\frac{k}2$};

\end{tikzpicture}   
\begin{tikzpicture}
      [scale=0.8,place/.style={circle,draw=black,thick,fill=black, inner sep=0pt,minimum size=1mm}]

\foreach \x in {0,1,2}{
      \node (\x) at ( 0, 1.5*\x) [place] {};
      \ifthenelse{\x>0}{\node (\x') at ( 4, 1.5*\x) [place] {}}{};
    };
    
\foreach \x in {0,1,2}      
\foreach \y in {1,2}
\ifthenelse{{{\x>1}\AND{\NOT{\x=\y}}}\OR{\x<2}}{\draw (\x) to (\y')}{};
\end{tikzpicture}  
\begin{tikzpicture}
      [scale=0.8,place/.style={circle,draw=black,thick,fill=black, inner sep=0pt,minimum size=1mm}]

\foreach \x in {0,1,2,3,4}{
      \node (\x) at ( 0, .9*\x) [place] {};
      \ifthenelse{\x>0}{\node (\x') at ( 4, .9*\x) [place] {}}{};
    };
    
\foreach \x in {0,1,2,3,4}      
\foreach \y in {1,2,3,4}
      \ifthenelse{{{\x>1}\AND{\NOT{\x=\y}}}\OR{\x<3}}{\draw (\x) to (\y')}{};
\end{tikzpicture}  
\caption{A $(2k+1,k^2+k/2)$-graph with girth $4$ on the left and cases $k=2,3$ on the middle and on the right. \label{girth4}}
\end{figure}

\begin{prop}
The maximum girth of a $k$-regular $(2k+1,k^2+k/2)$-graph is $5$ if $k=2$ and  $3$ if  $k\geq 4$. Besides, for any $k$ there are $(2k+1,k^2+k/2)$-graphs with girth $4$.
\end{prop}
As a consequence either we have a negative answer to Question~\ref{question} or we have an infinite many counterexamples to Boesch conjecture on the girth.
\begin{prop}\label{Prop:either}
Either there is no UMRG $(2k+1, k^2+k/2)$-graph or it has not maximum girth.
\end{prop}
It is interesting to remark that although the $(9,18)$-graph in Figure~\ref{n9k4} does not  have maximum girth among the $(9,18)$-graphs, \emph{it does have minimum number of 3-cycles among the $4$-regular $(9,18)$-graphs}, as can be checked by hand.

When exploring UMRGs energy, a pattern was not found among these graphs. While in some classes reliable graphs present the lowest possible energy ($n=6$ and $e=9,12$; $n=8$ and $e=16,24$), other classes exhibit intermediate values. To compare energy between optimal graphs and non-optimal graphs, we estimate energy for every graph in each class and  computed the percentile of the UMRG energy. The $(7,14)$-UMRG energy percentile was 15\%; $(8,12)$-UMRG energy percentile was 66\%; $(8,20)$-UMRG energy percentile was 15 and $(9,18)$-UMRG energy percentile was 5\%. It is clear that there is no conspicuous pattern in energy values among UMRG.  We found a similar behaviour when assessing Laplacian energy. While the vast majority of UMRG attained the minimum value within their classes, $(8,12)$, $(8,20)$ and $(9,18)$ classes attained values corresponding to the percentiles: 3\%, 1\% and 0.02\%. Despite being very low values we cannot affirm that UMRG share a common characteristic.

Except for $(7,14)$-optimal graph, every UMRG attained the maximum Fiedler number. In the former case, the Fiedler number corresponded to the percentile 98\%. Although high, not the maximum as in the other cases.

\subsection{UMRG candidates of higher order}
For computational reasons, only coefficient $N_t$, the number of spanning trees, was computed for $k$-regular graphs with $10 \leq n\leq14$. Again, a single graph maximizing $N_t$ could be found in each class. We choose this graph as a candidate to UMRG although proofs must be done in order to ascertain this fact. 

In Figure~\ref{umrgn10},\ref{umrgn11},\ref{umrgn12}, \ref{umrgn13} and \ref{umrgn14} are depicted the potential UMRG with 10, 11, 12, 13 and 14 nodes respectively. Complete bipartite graphs were found to be candidates to UMRG in $(10,25)$,$(12,36)$ and $(14,49)$ classes. As proved in literature~\cite{petersen}, Petersen graph is the optimal graph among $(10,15)$-graphs (Fig.~\ref{n10k3}). Regarding graphs with order 14, coincidentally with Ath \& Sobel~\cite{ath_sobel_conj}, Heawood graph is $t$-optimal among $(14,21)$-graphs. Others graphs are shown in Figure~\ref{umrgn14}. 

\begin{figure}
	\centering
	\begin{subfigure}[b]{0.3\textwidth}
		\centering
		\includegraphics[width=\textwidth]{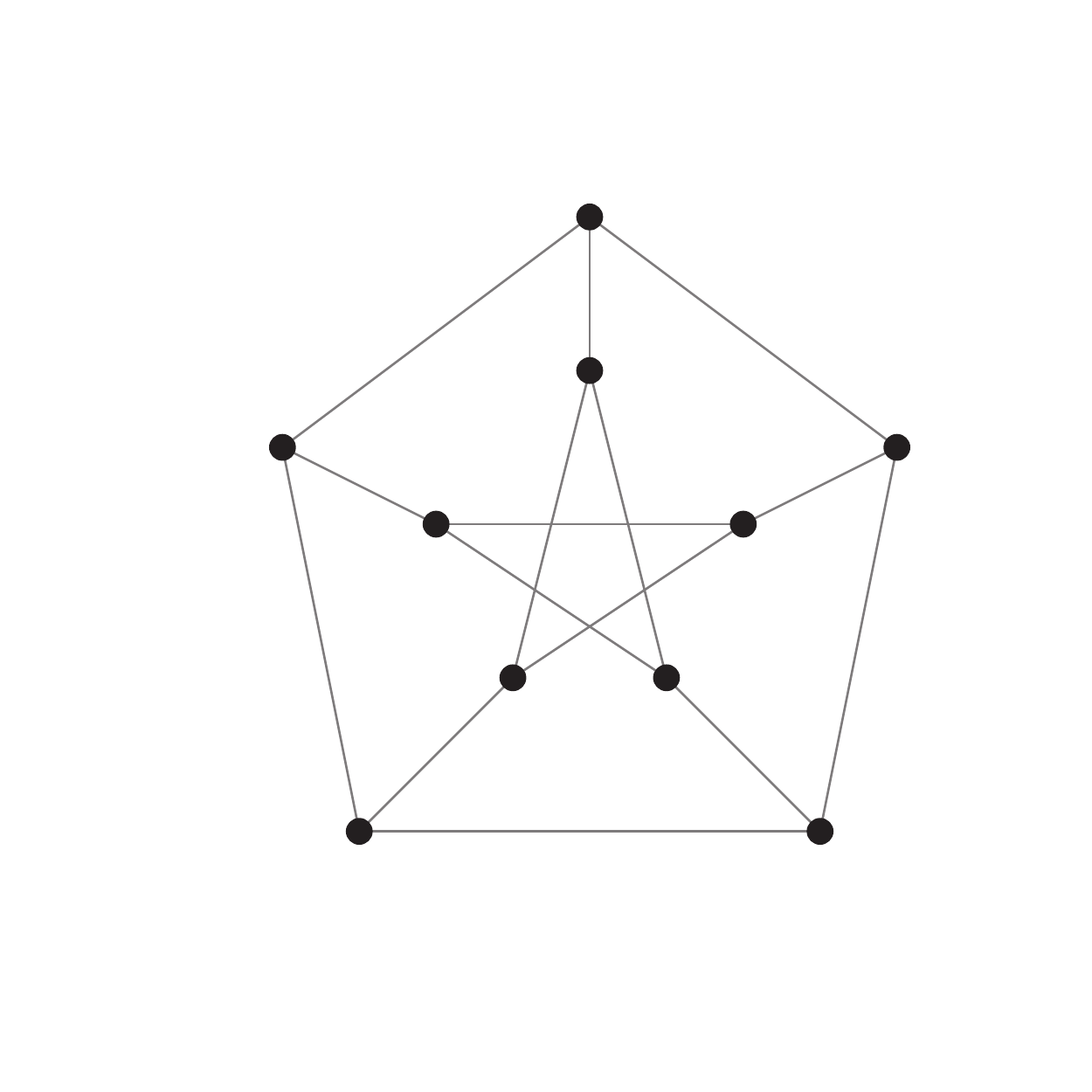}
		\caption{$n=10$, $e=15$ (Petersen graph) \\ $N_t=2000$}
		\label{n10k3}
	\end{subfigure}
	\hfill
	\begin{subfigure}[b]{0.3\textwidth}
		\centering
		\includegraphics[width=\textwidth]{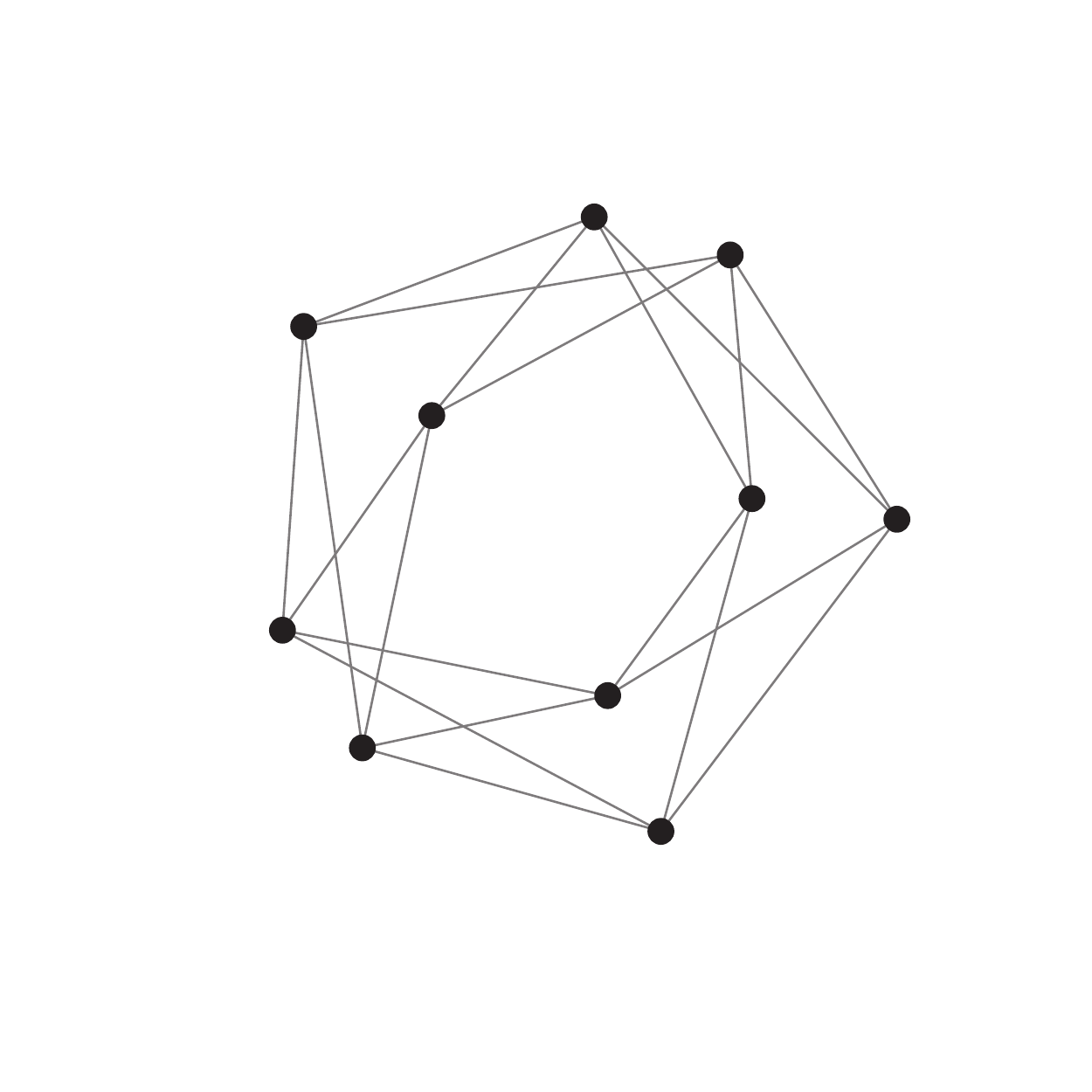}
		\caption{$n=10$, $e=20$ \\$N_t=40960$}
		\label{n10k4}
	\end{subfigure}
	\hfill
	\begin{subfigure}[b]{0.3\textwidth}
		\centering
		\includegraphics[width=\textwidth]{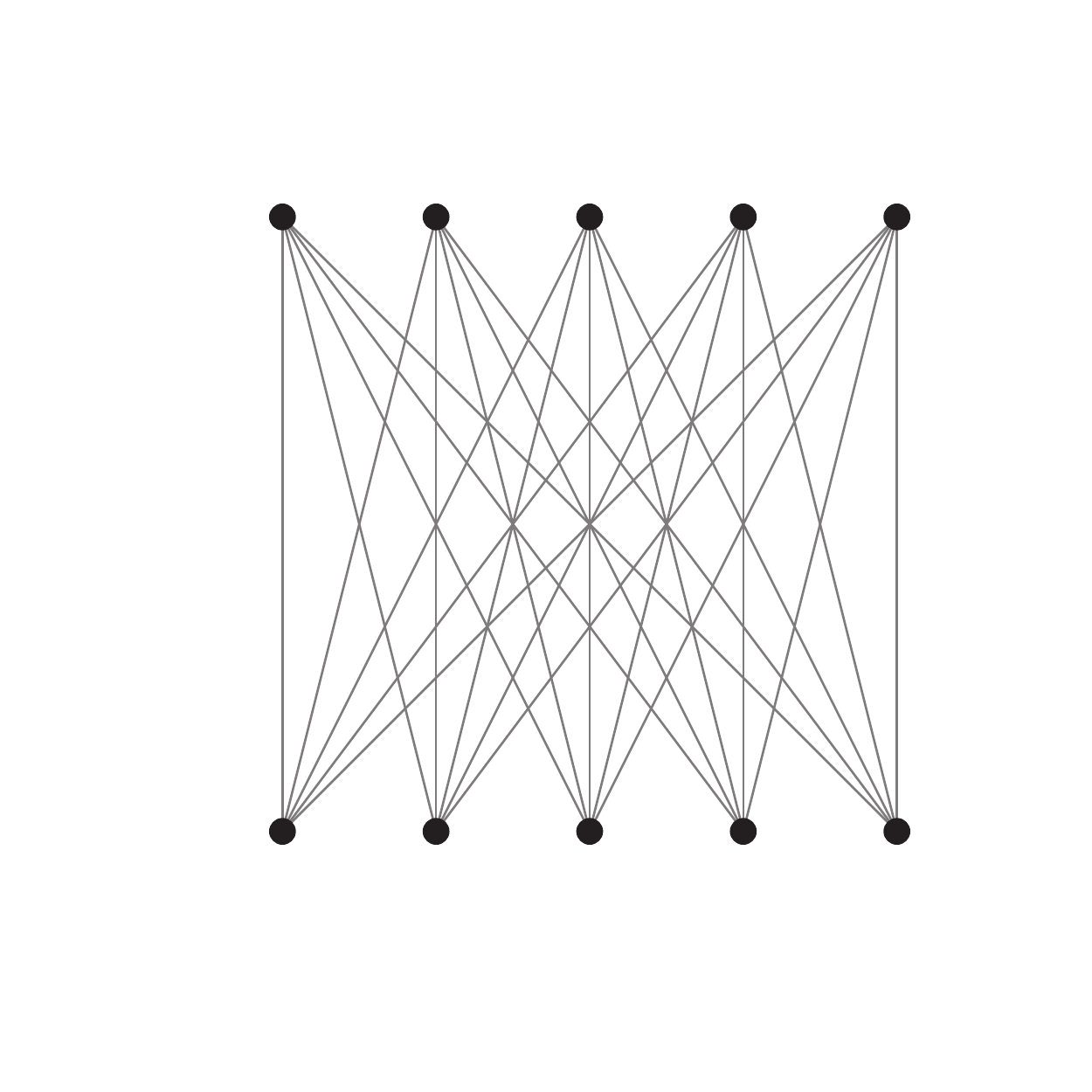}
		\caption{$n=10$, $e=25$ (graph $K_{5,5}$)\\­$N_t=390625$}
		\label{n10k5}
	\end{subfigure}
	\hfill
	\begin{subfigure}[b]{0.3\textwidth}
		\centering
		\includegraphics[width=\textwidth]{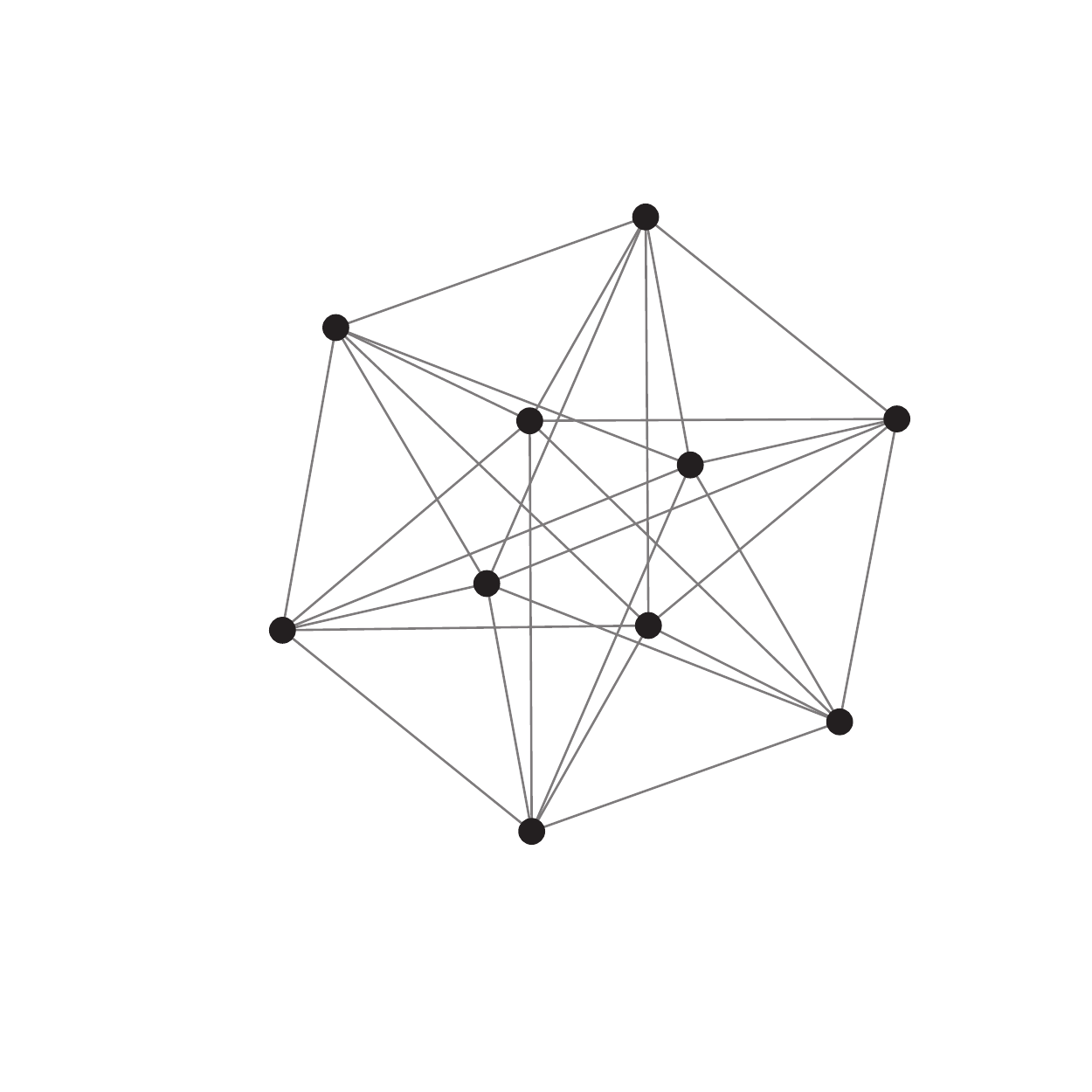}
		\caption{$n=10$, $e=30$\\ $N_t=2116800$}
		\label{n10k6}
	\end{subfigure}
	\hfill
	\begin{subfigure}[b]{0.3\textwidth}
		\centering
		\includegraphics[width=\textwidth]{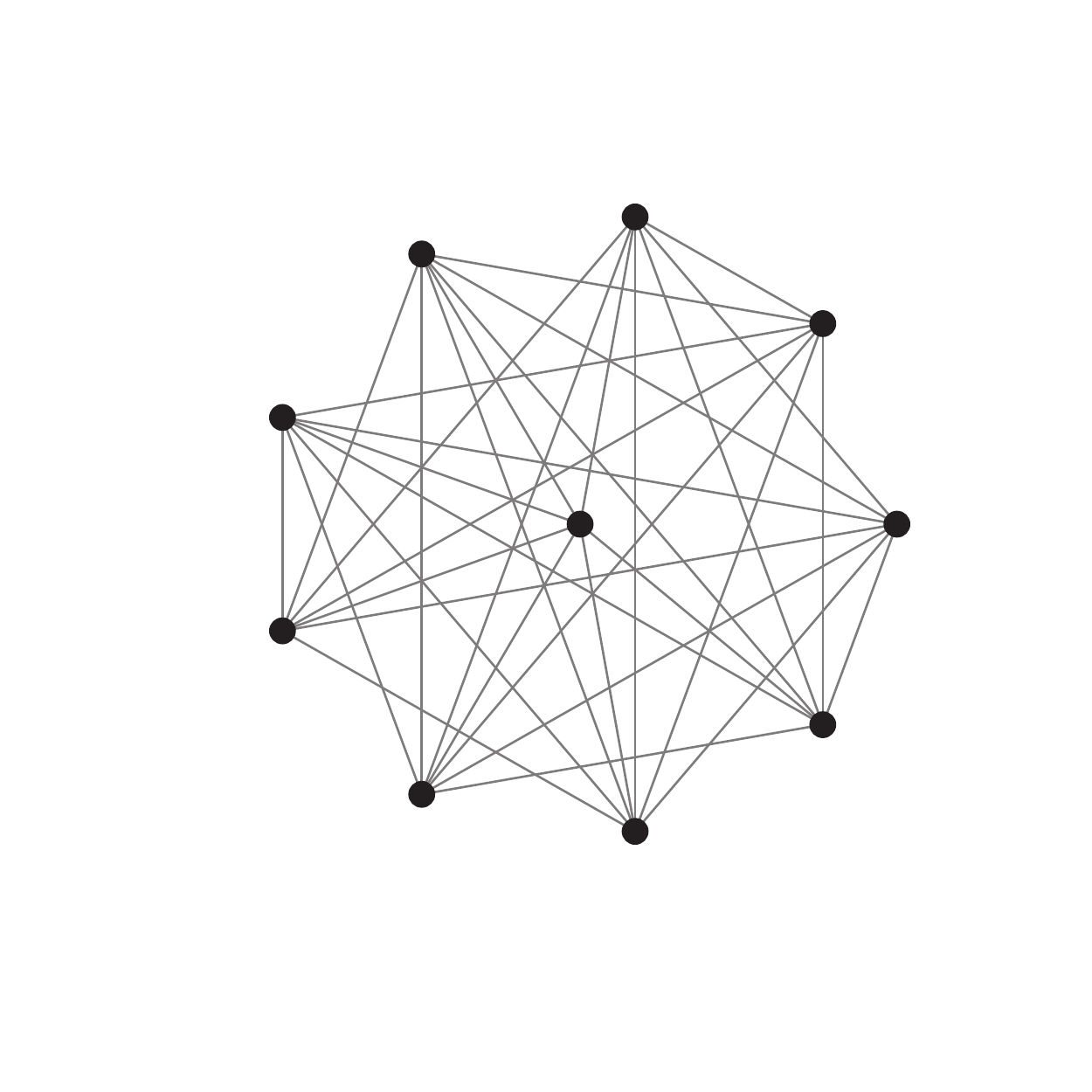}
		\caption{$n=10$, $e=35$ (graph $\overline{2C_3\cup C_4}$)\\ $N_t=9219840$}
		\label{n10k7}
	\end{subfigure}
	\caption{Potential UMRG with $n=10$}
	\label{umrgn10}
\end{figure}

\begin{figure}
	\centering
	\begin{subfigure}[b]{0.3\textwidth}
		\centering
		\includegraphics[width=\textwidth]{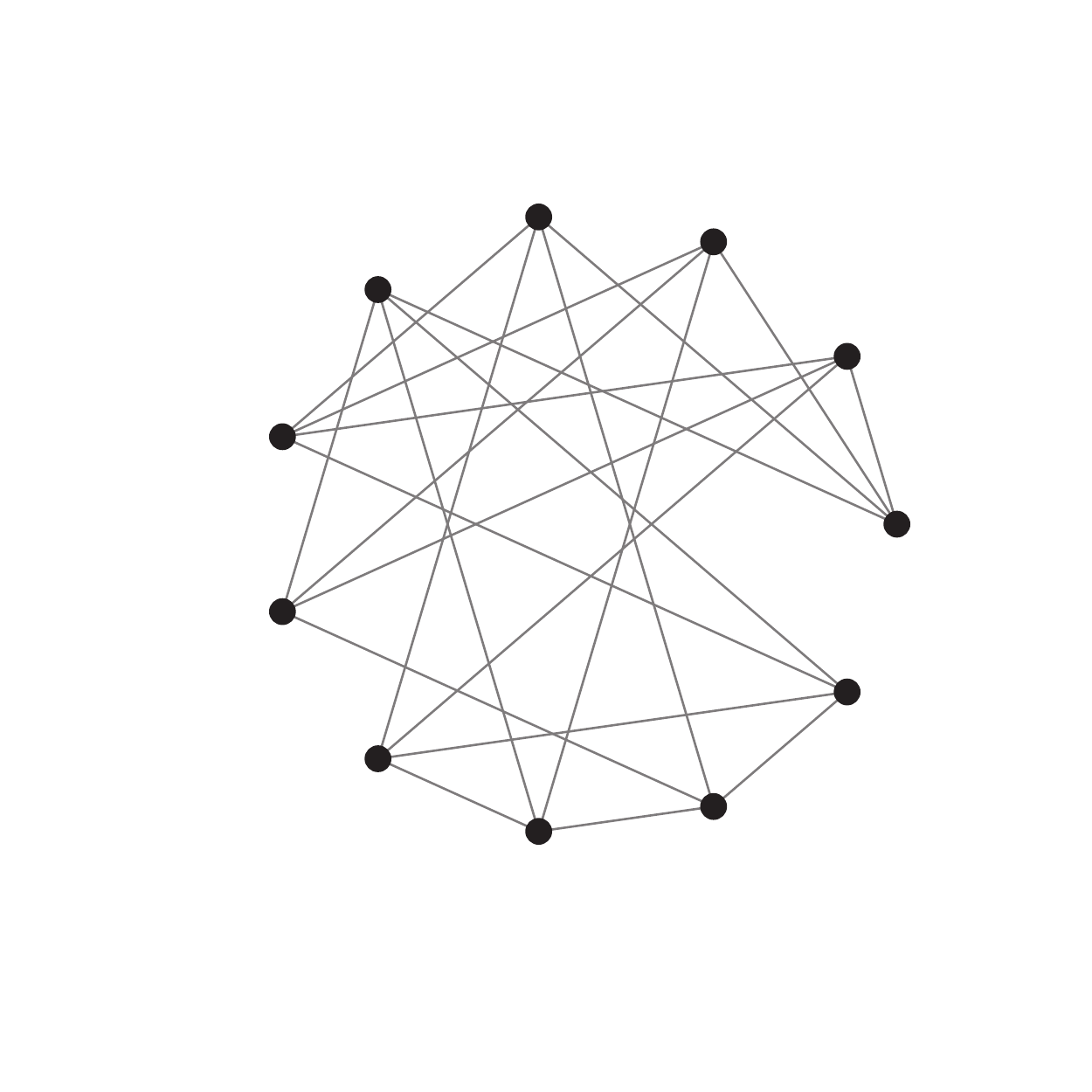}
		\caption{$n=11$, $e=22$\\ $N_t=130691$}
		\label{n11k4}
	\end{subfigure}
	\hfill
	\begin{subfigure}[b]{0.3\textwidth}
		\centering
		\includegraphics[width=\textwidth]{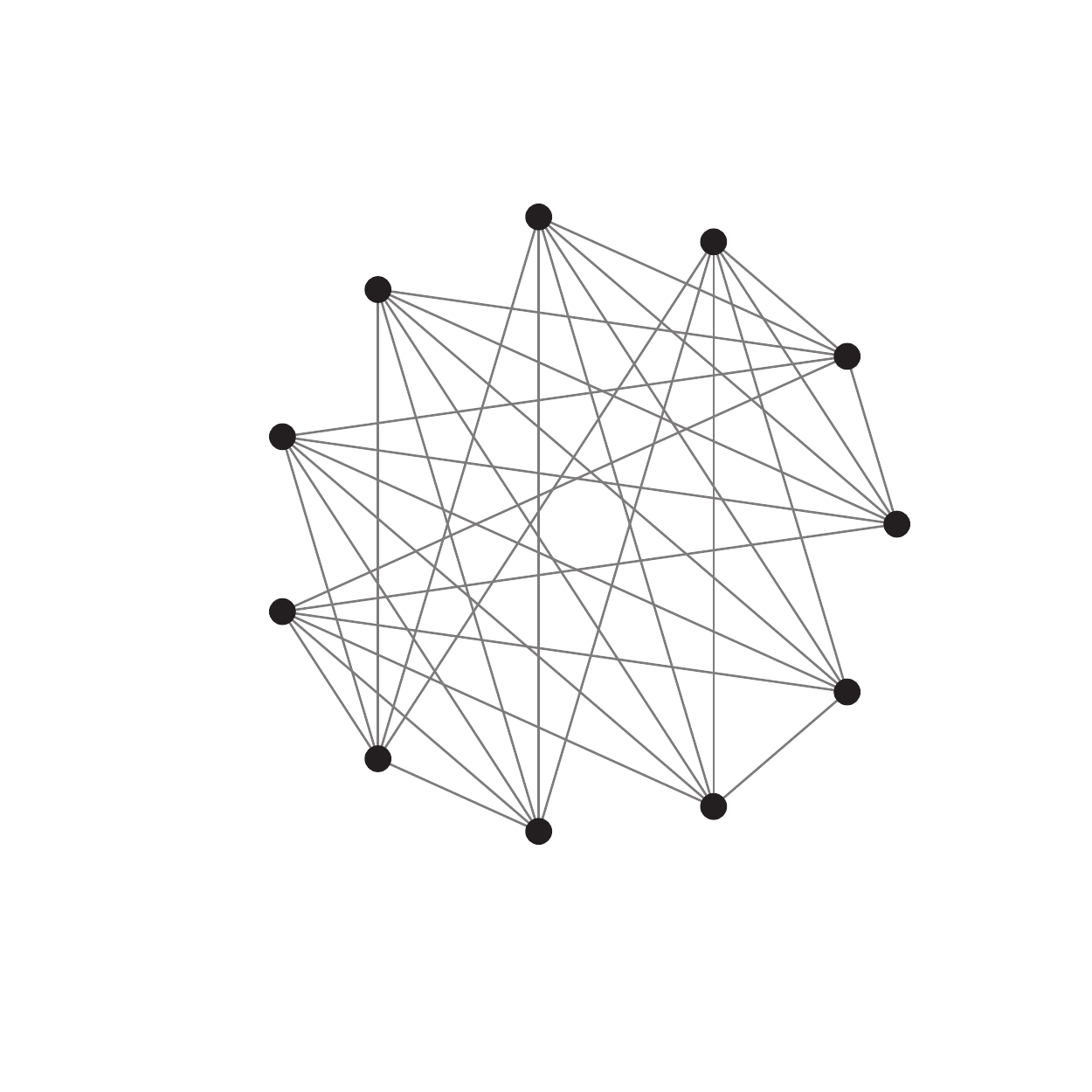}
		\caption{$n=11$, $e=33$\\ $N_t=11113200$}
		\label{n11k6}
	\end{subfigure}
	\hfill
	\begin{subfigure}[b]{0.3\textwidth}
		\centering
		\includegraphics[width=\textwidth]{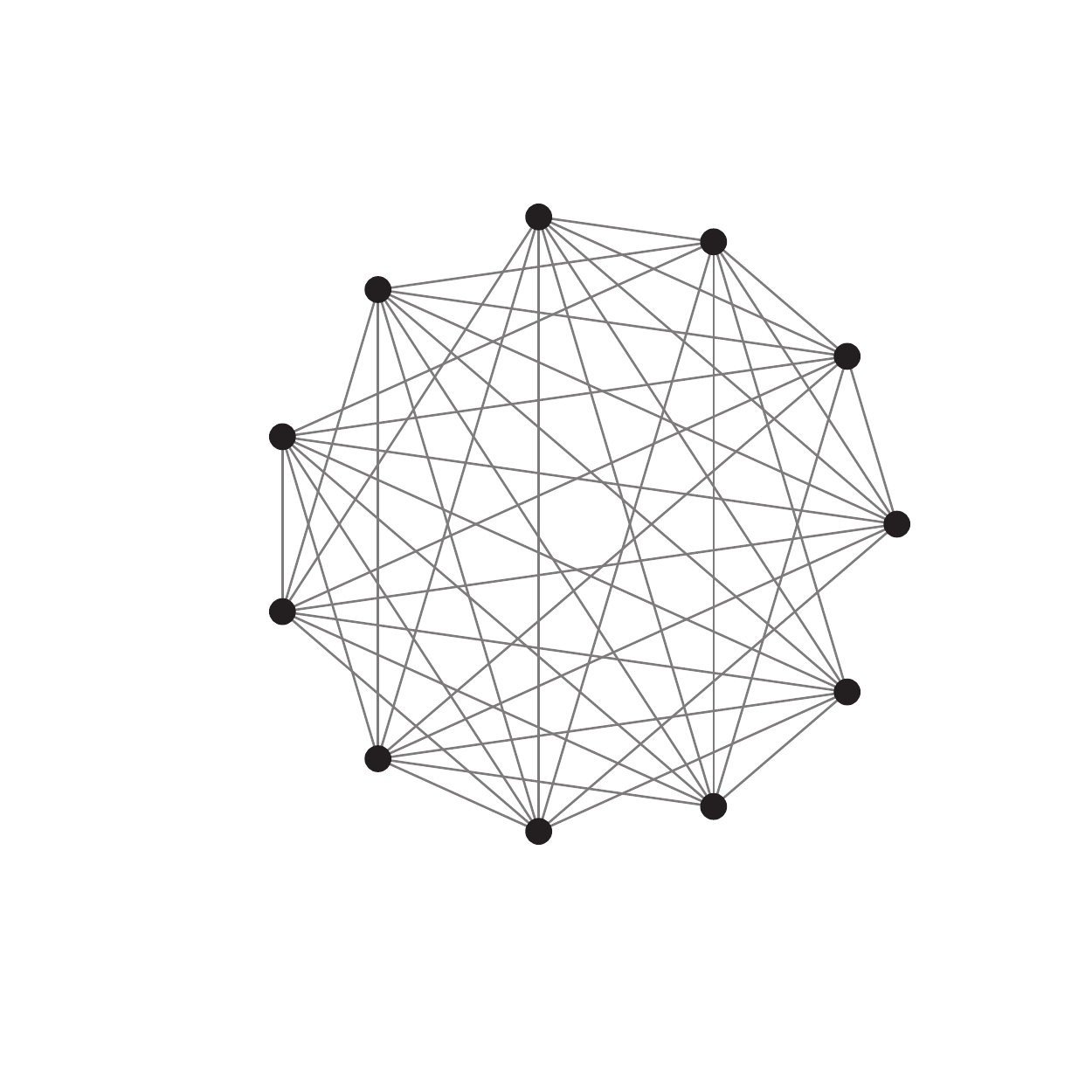}
		\caption{$n=11$, $e=44$ (graph $\overline{2C_3\cup C_5}$)\\ $N_t=227127296$}
		\label{n11k8}
	\end{subfigure}
	\hfill
	\caption{Potential UMRG with $n=11$}
	\label{umrgn11}
\end{figure}

\begin{figure}
	\centering
	\begin{subfigure}[b]{0.3\textwidth}
		\centering
		\includegraphics[width=\textwidth]{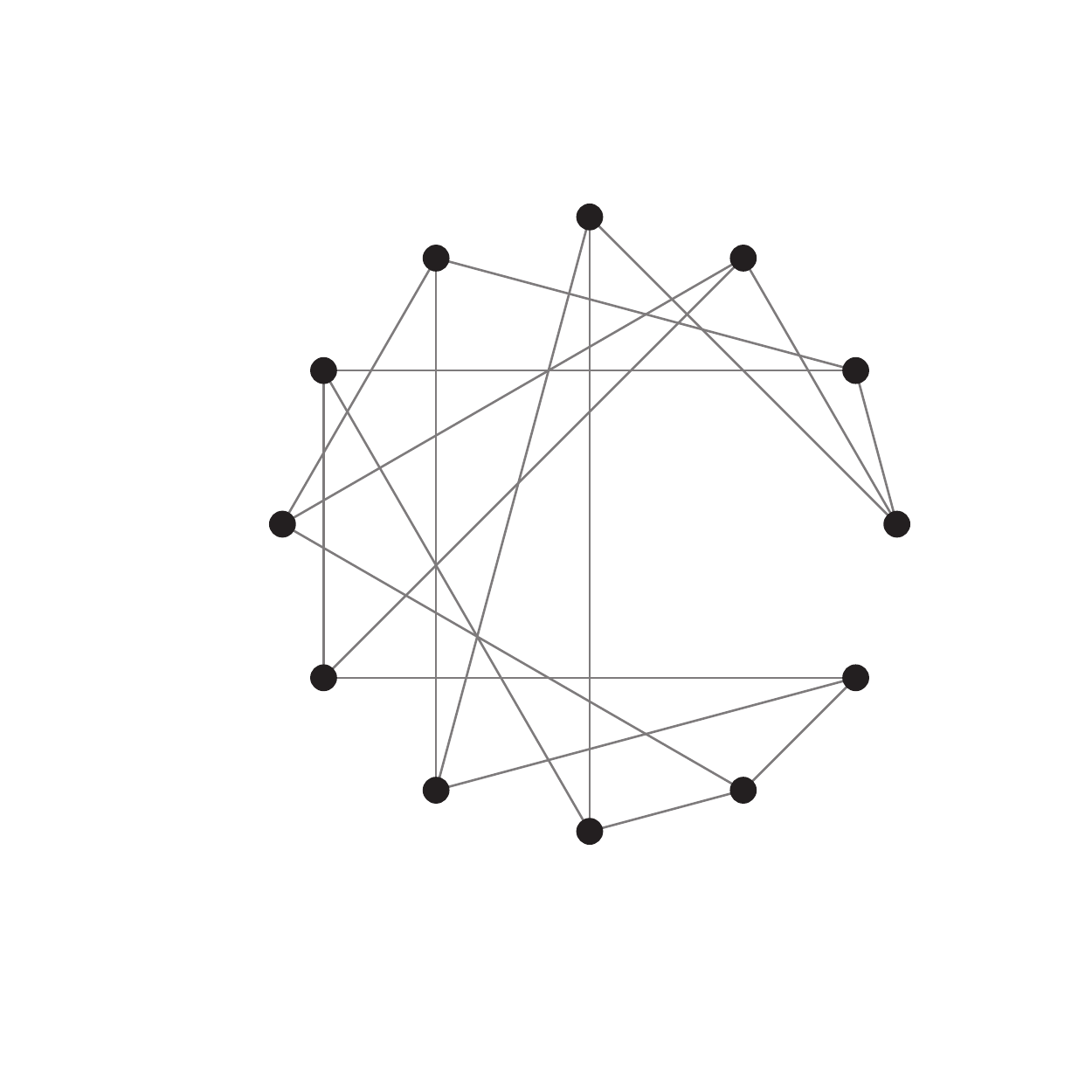}
		\caption{$n=12$, $e=18$\\ $N_t=9800$}
		\label{n12k3}
	\end{subfigure}
	\hfill
	\begin{subfigure}[b]{0.3\textwidth}
		\centering
		\includegraphics[width=\textwidth]{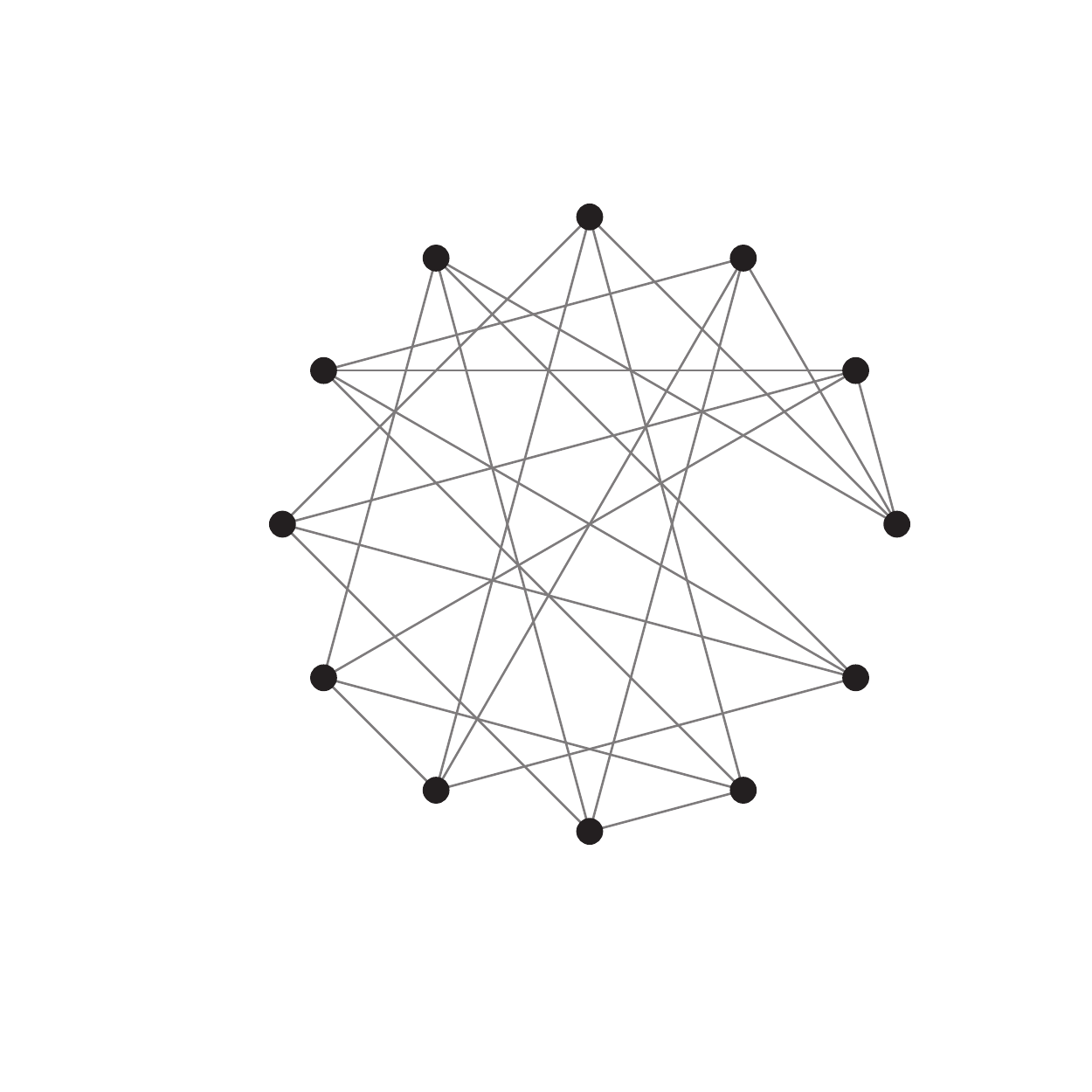}
		\caption{$n=12$, $e=24$\\ $N_t=428652$}
		\label{n12k4}
	\end{subfigure}
	\hfill
	\begin{subfigure}[b]{0.3\textwidth}
		\centering
		\includegraphics[width=\textwidth]{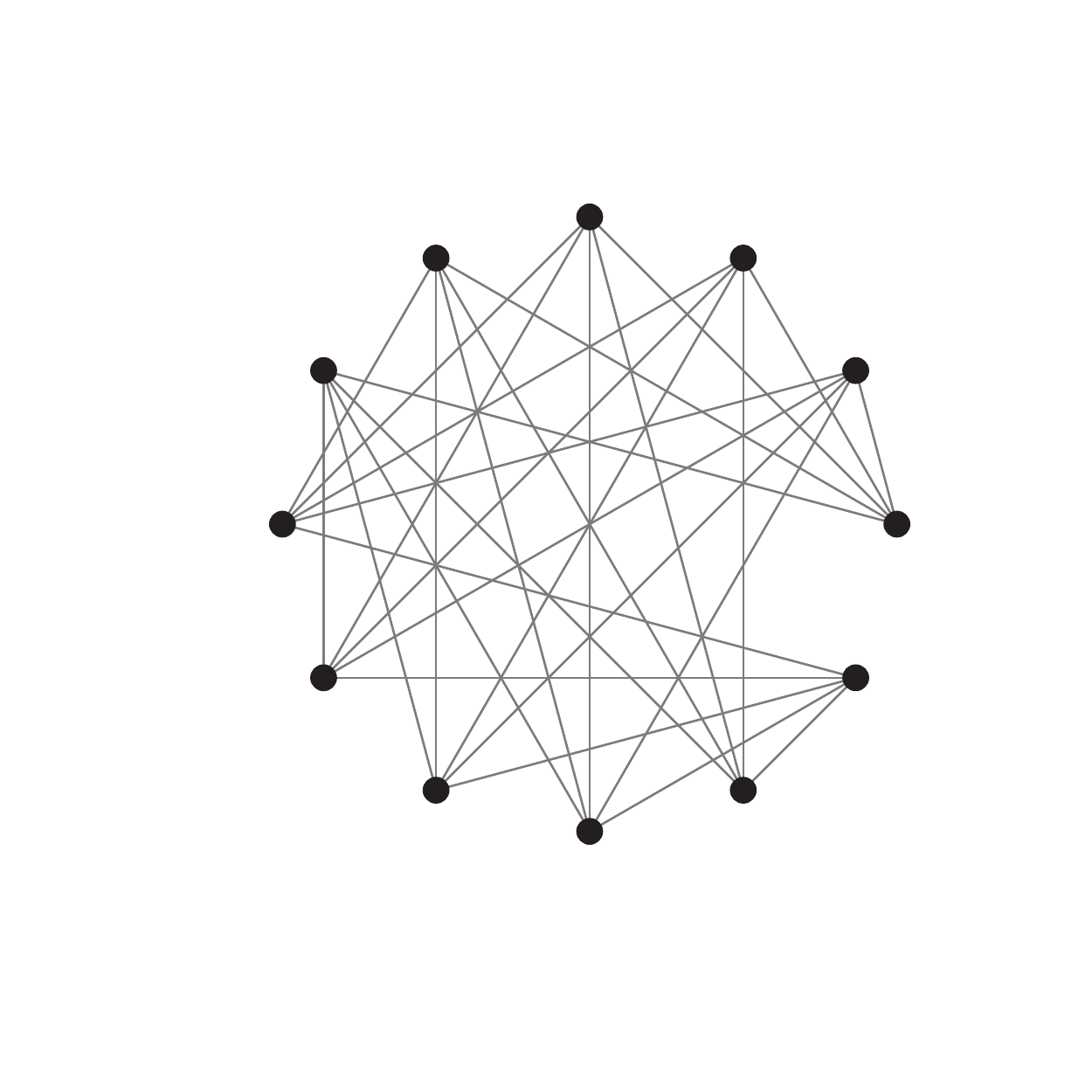}
		\caption{$n=124$, $e=30$\\ $N_t=6635520$}
		\label{n12k5}
	\end{subfigure}
	\hfill
		\begin{subfigure}[b]{0.3\textwidth}
		\centering
		\includegraphics[width=\textwidth]{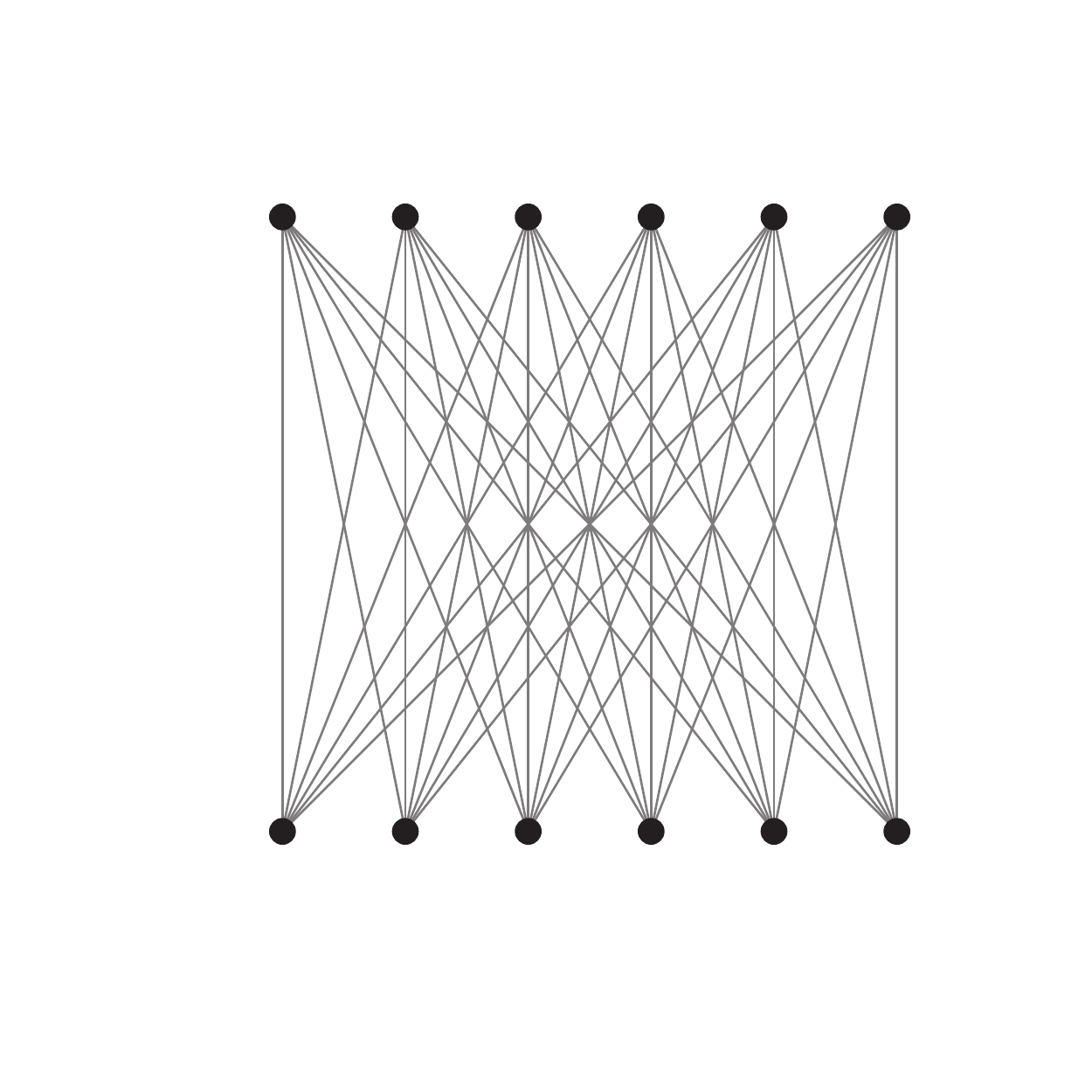}
		\caption{$n=12$, $e=36$
		(graph $K_{6,6}$)\\ $N_t=60466176$}
		\label{n12k6}
	\end{subfigure}
	\hfill
			\begin{subfigure}[b]{0.3\textwidth}
		\centering
		\includegraphics[width=\textwidth]{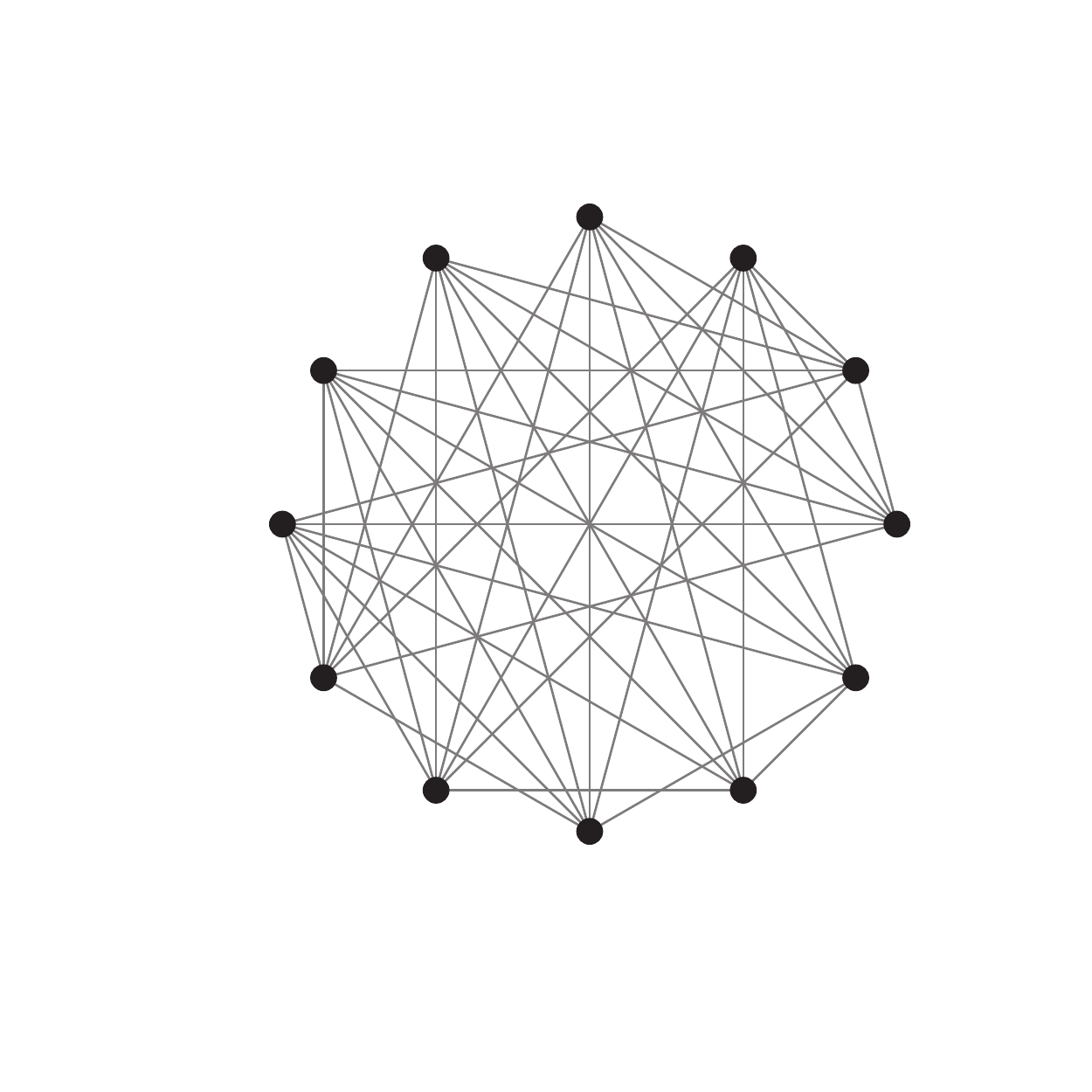}
		\caption{$n=12$, $e=42$ \\ $N_t=341251729$}
		\label{n12k7}
	\end{subfigure}
	\hfill
				\begin{subfigure}[b]{0.3\textwidth}
		\centering
		\includegraphics[width=\textwidth]{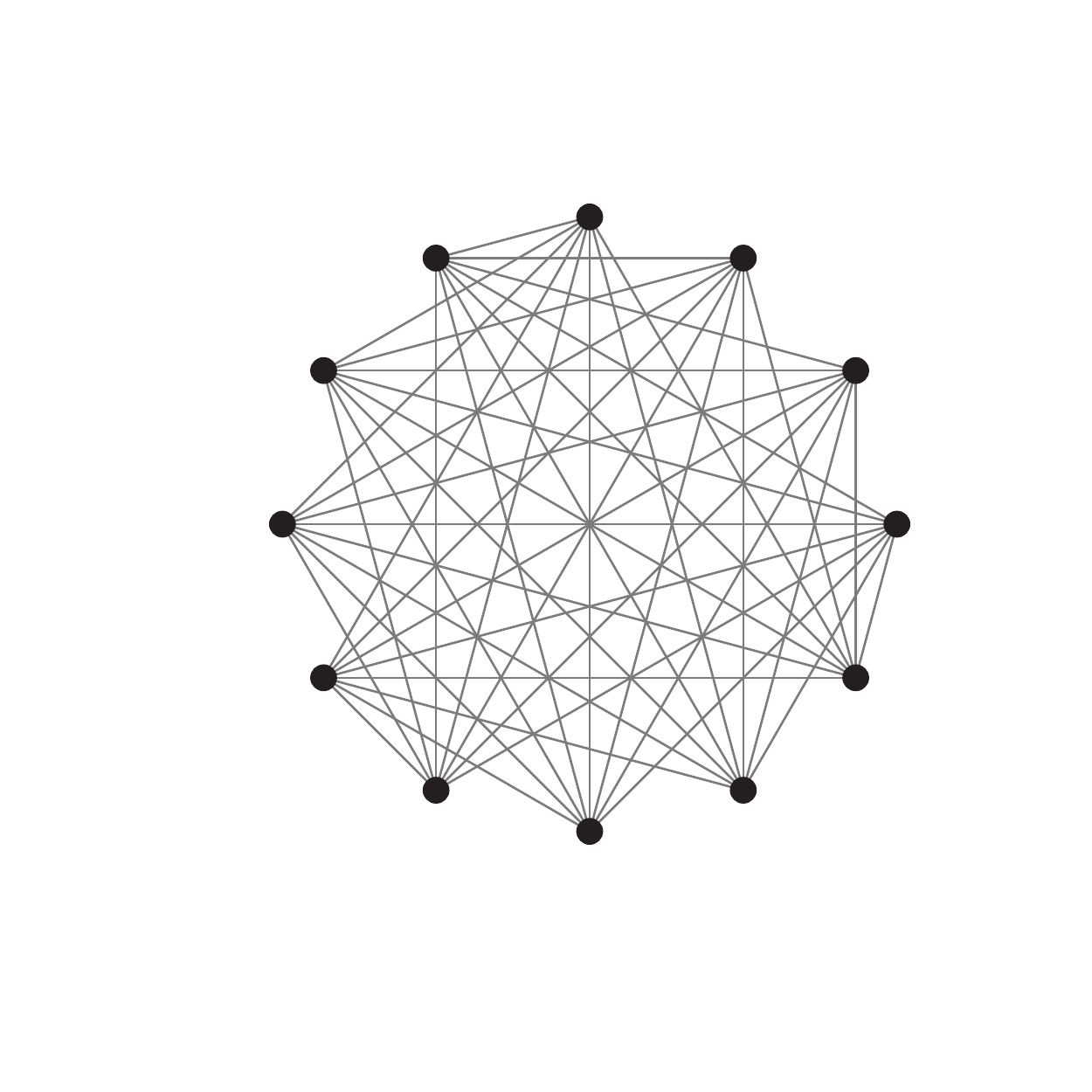}
		\caption{$n=12$, $e=48$ \\ $N_t=1610612736$}
		\label{n12k8}
	\end{subfigure}
	\hfill
	\begin{subfigure}[b]{0.3\textwidth}
		\centering
		\includegraphics[width=\textwidth]{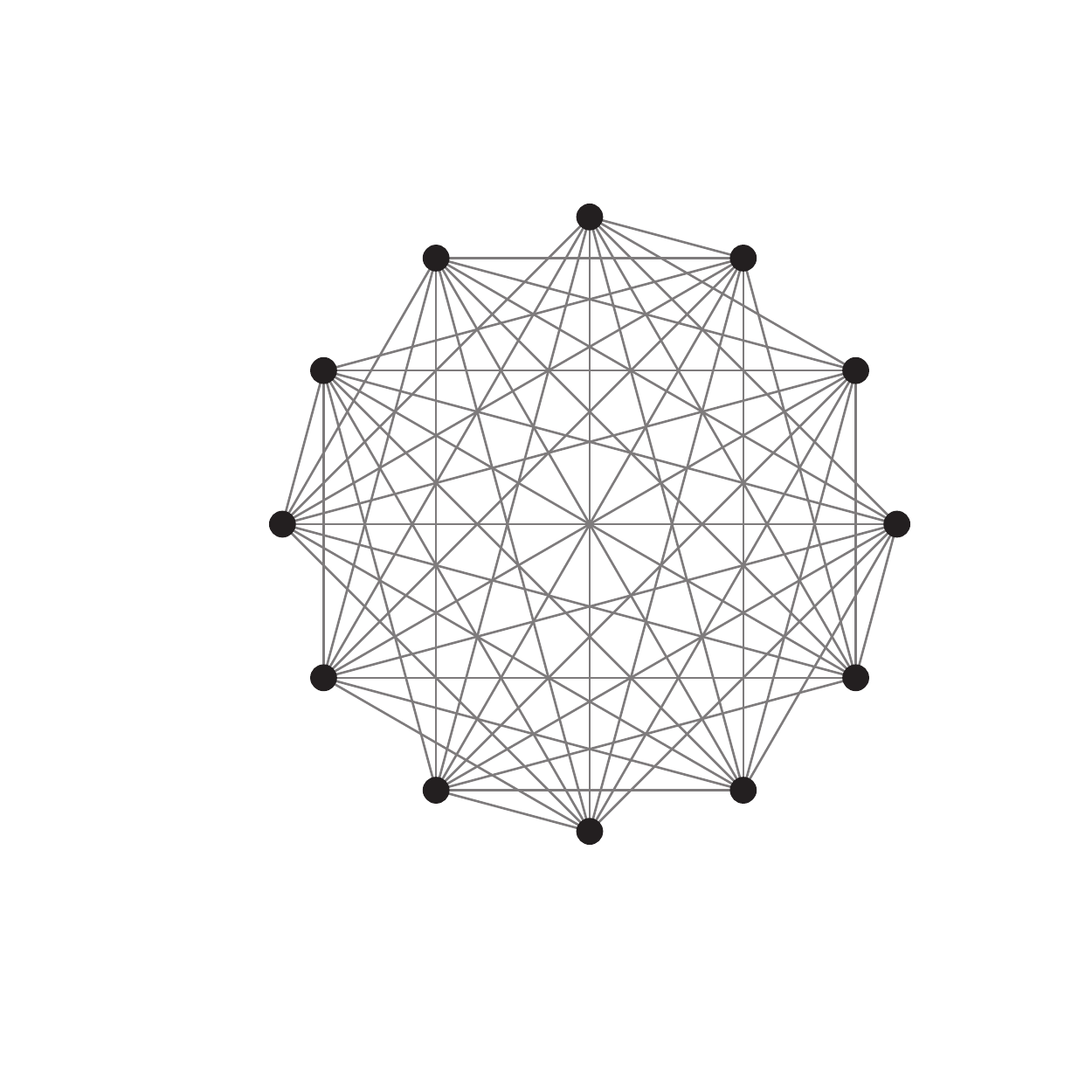}
		\caption{$n=12$, $e=54$ (graph $\overline{4C_3}$)\\$N_t=6198727824$}
		\label{n12k9}
	\end{subfigure}
	\caption{Potential UMRG with $n=12$}
	\label{umrgn12}
\end{figure}

\begin{figure}
	\centering
	\begin{subfigure}[b]{0.3\textwidth}
		\centering
		\includegraphics[width=\textwidth]{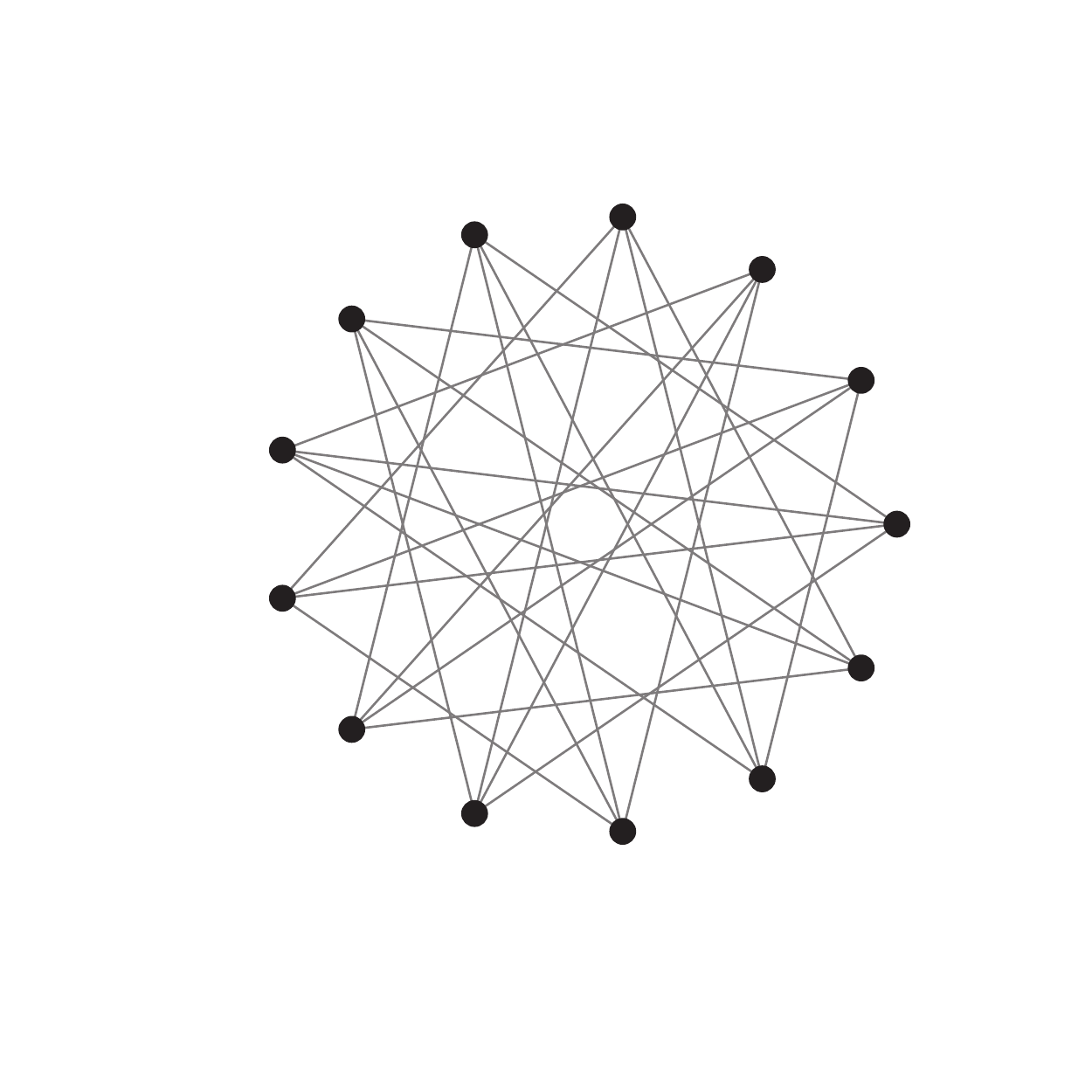}
		\caption{$n=13$, $e=26$\\ $N_t=1373125$}
		\label{n13k4}
	\end{subfigure}
	\hfill
	\begin{subfigure}[b]{0.3\textwidth}
		\centering
		\includegraphics[width=\textwidth]{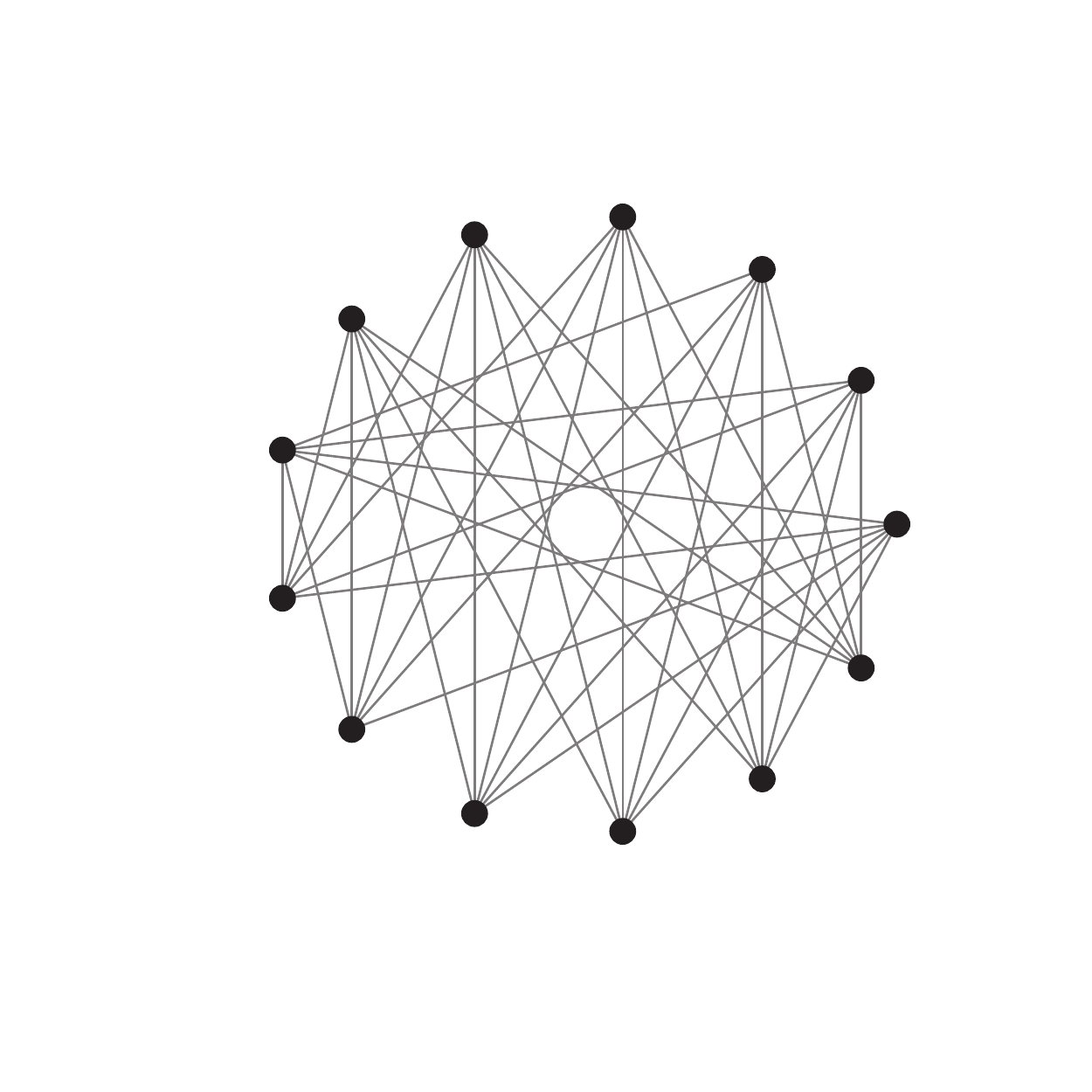}
		\caption{$n=13$, $e=39$\\ $N_t=300056400$}
		\label{n13k6}
	\end{subfigure}
	\hfill
	\begin{subfigure}[b]{0.3\textwidth}
		\centering
		\includegraphics[width=\textwidth]{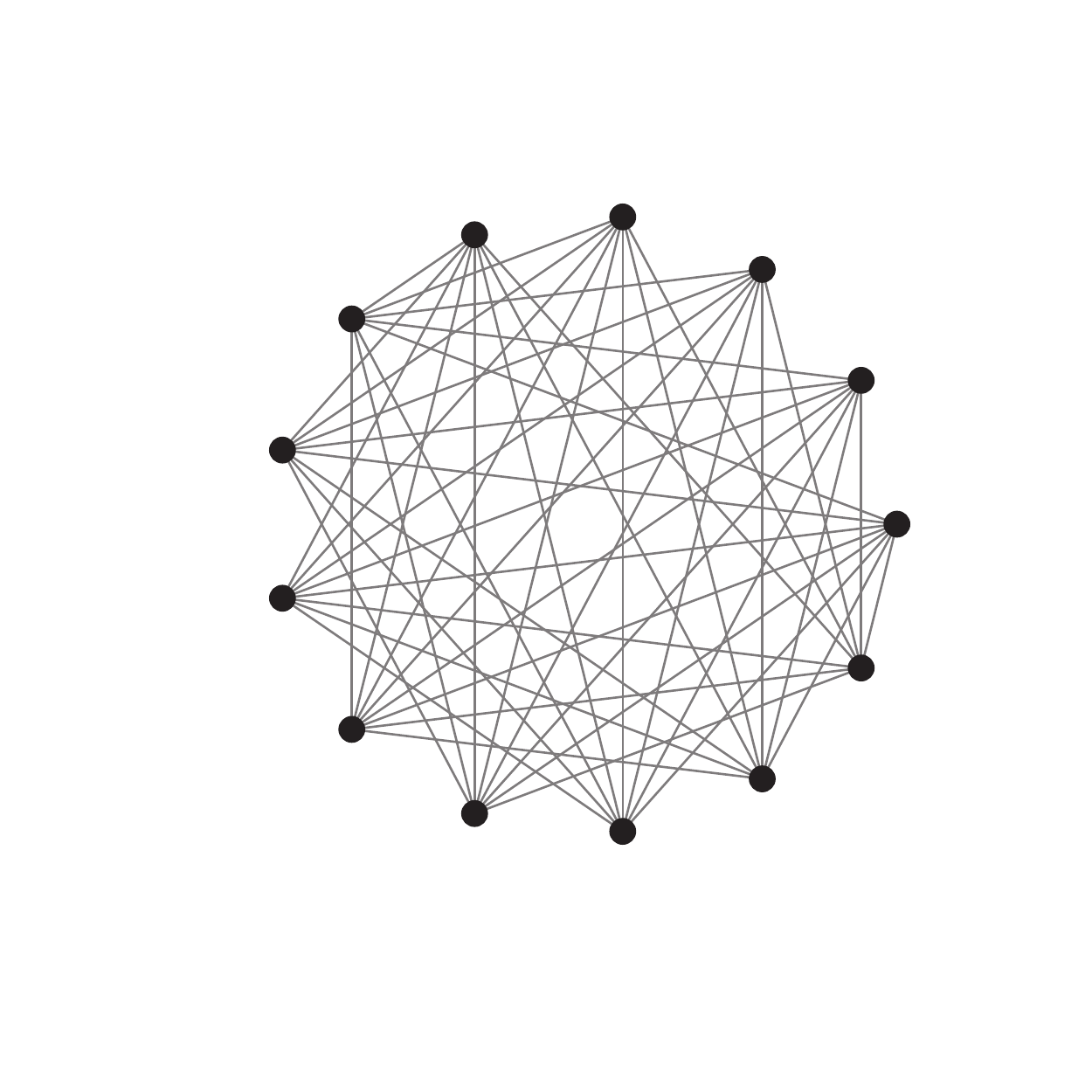}
		\caption{$n=13$, $e=52$\\ $N_t\approx 1.12\times10^{10}$}
		\label{n13k8}
	\end{subfigure}
	\hfill
	\begin{subfigure}[b]{0.3\textwidth}
		\centering
		\includegraphics[width=\textwidth]{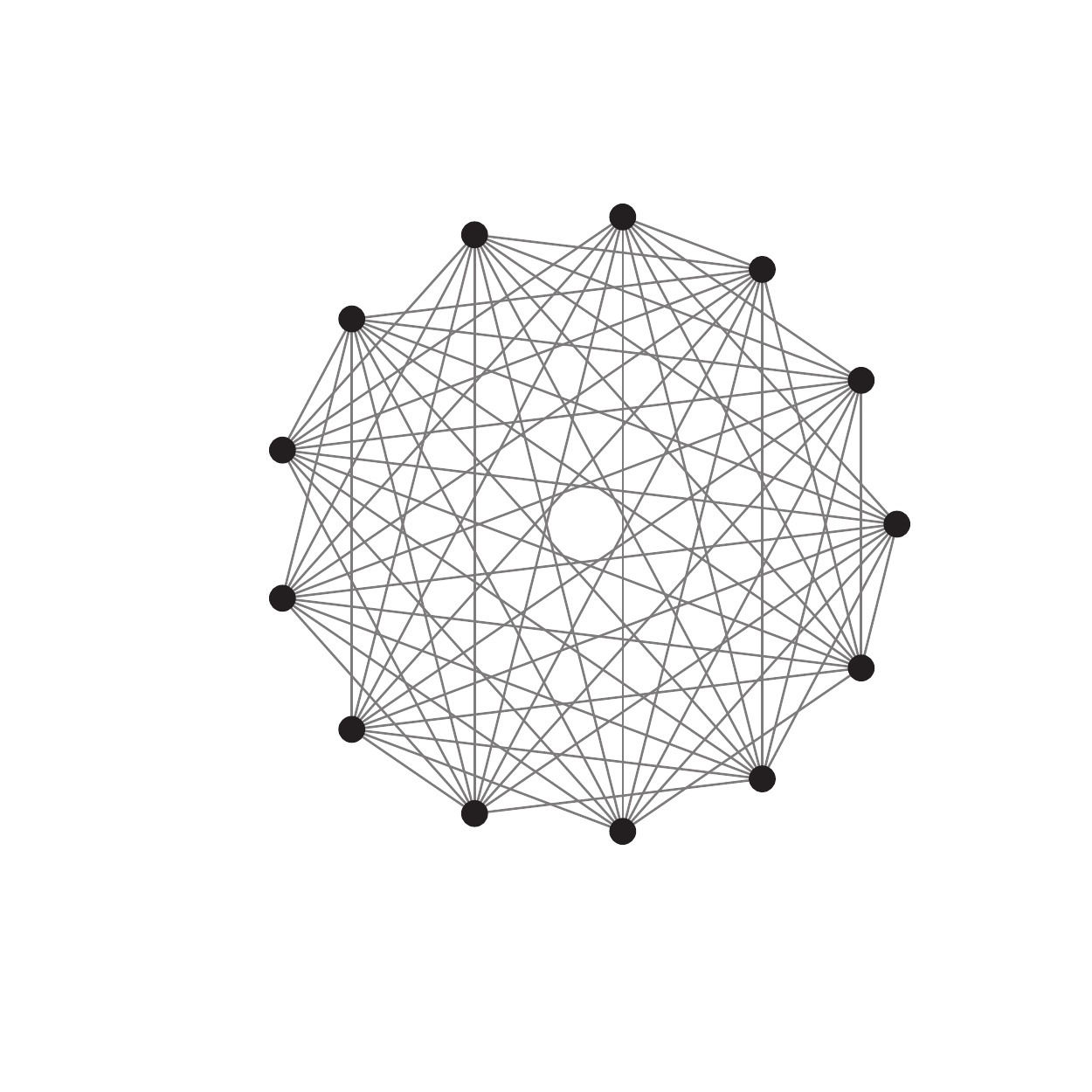}
		\caption{$n=13$, $e=65$ (graph $\overline{3C_3\cup C_4}$) \\ $N_t\approx1.84\times10^{11}$}
		\label{n13k10}
	\end{subfigure}
	\hfill
	\caption{Potential UMRG with $n=13$}
	\label{umrgn13}
\end{figure}

\begin{figure}
	\centering
	\begin{subfigure}[b]{0.3\textwidth}
		\centering
		\includegraphics[width=\textwidth]{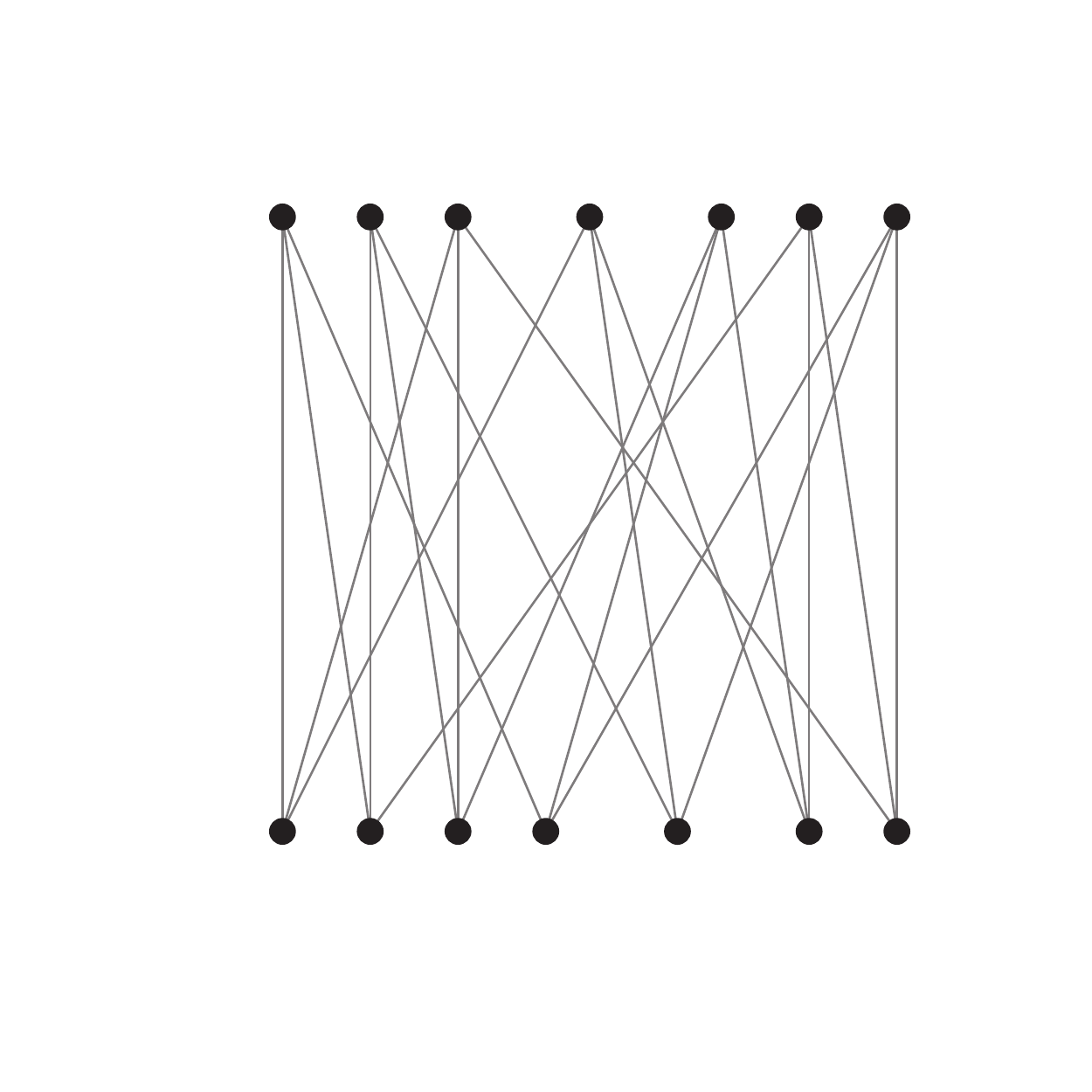}
		\caption{$n=14$, $e=21$\\ $N_t=50421$}
		\label{n14k3}
	\end{subfigure}
	\hfill
	\begin{subfigure}[b]{0.3\textwidth}
		\centering
		\includegraphics[width=\textwidth]{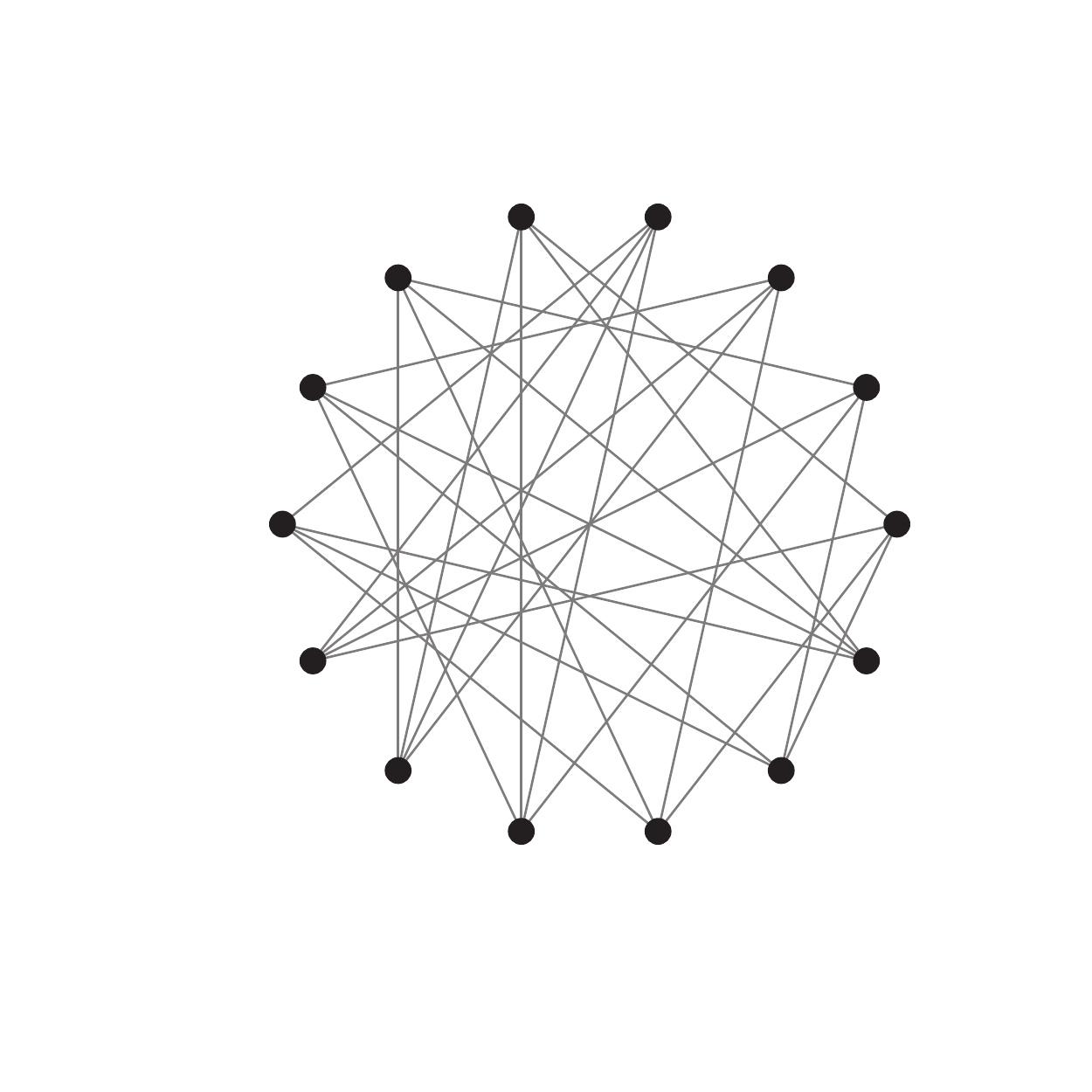}
		\caption{$n=14$, $e=28$\\ $N_t=4423680$}
		\label{n14k4}
	\end{subfigure}
	\hfill
	\begin{subfigure}[b]{0.3\textwidth}
		\centering
		\includegraphics[width=\textwidth]{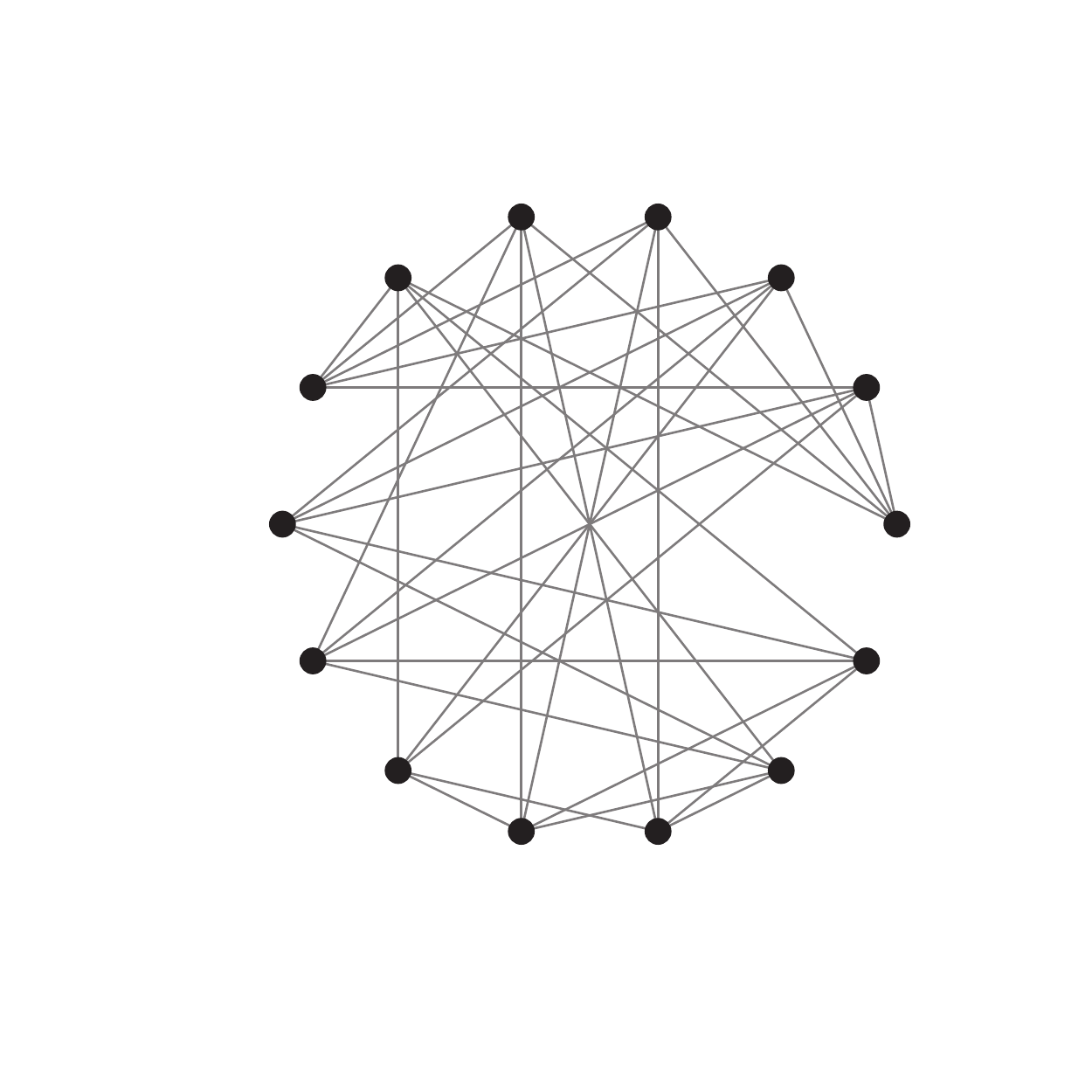}
		\caption{$n=14$, $e=35$\\ $N_t=116280000$}
		\label{n14k5}
	\end{subfigure}
	\hfill
		\begin{subfigure}[b]{0.3\textwidth}
		\centering
		\includegraphics[width=\textwidth]{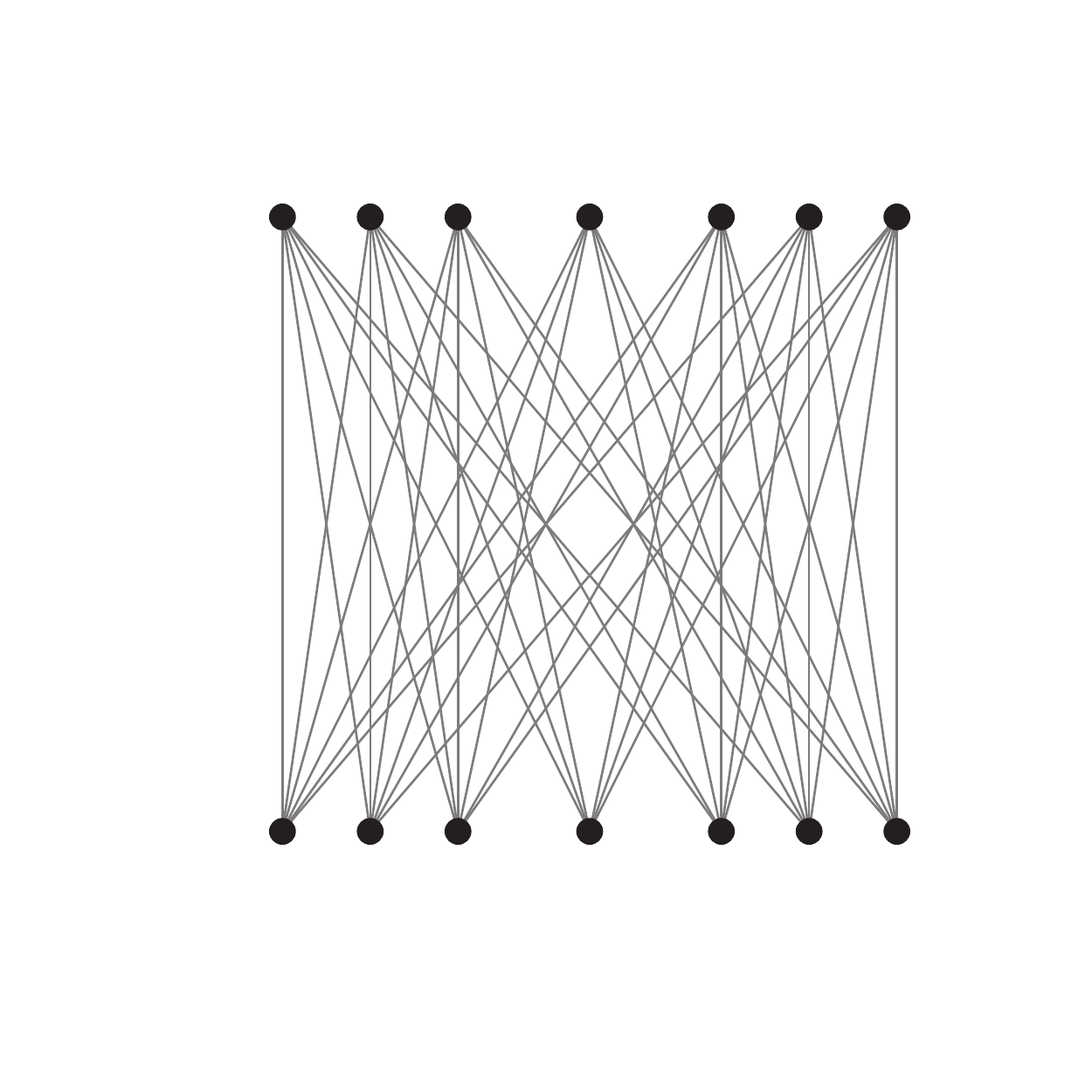}
		\caption{$n=14$, $e=42$\\ $N_t=1575656250$}
		\label{n14k6}
	\end{subfigure}
		\hfill
		\begin{subfigure}[b]{0.3\textwidth}
		\centering
		\includegraphics[width=\textwidth]{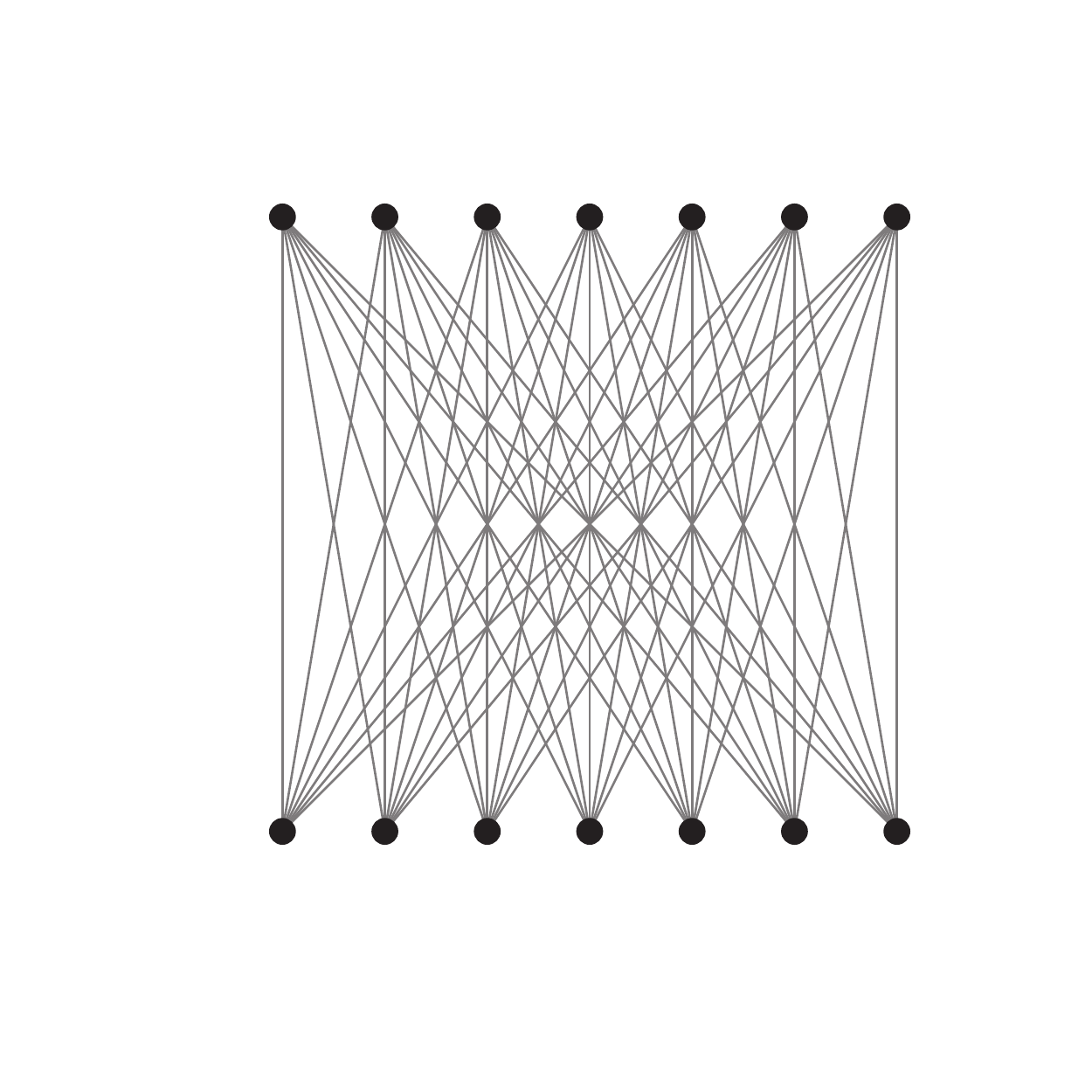}
		\caption{$n=14$, $e=49$\\ $N_t\approx1.38\times10^{10}$}
		\label{n14k7}
	\end{subfigure}
		\hfill
		\begin{subfigure}[b]{0.3\textwidth}
		\centering
		\includegraphics[width=\textwidth]{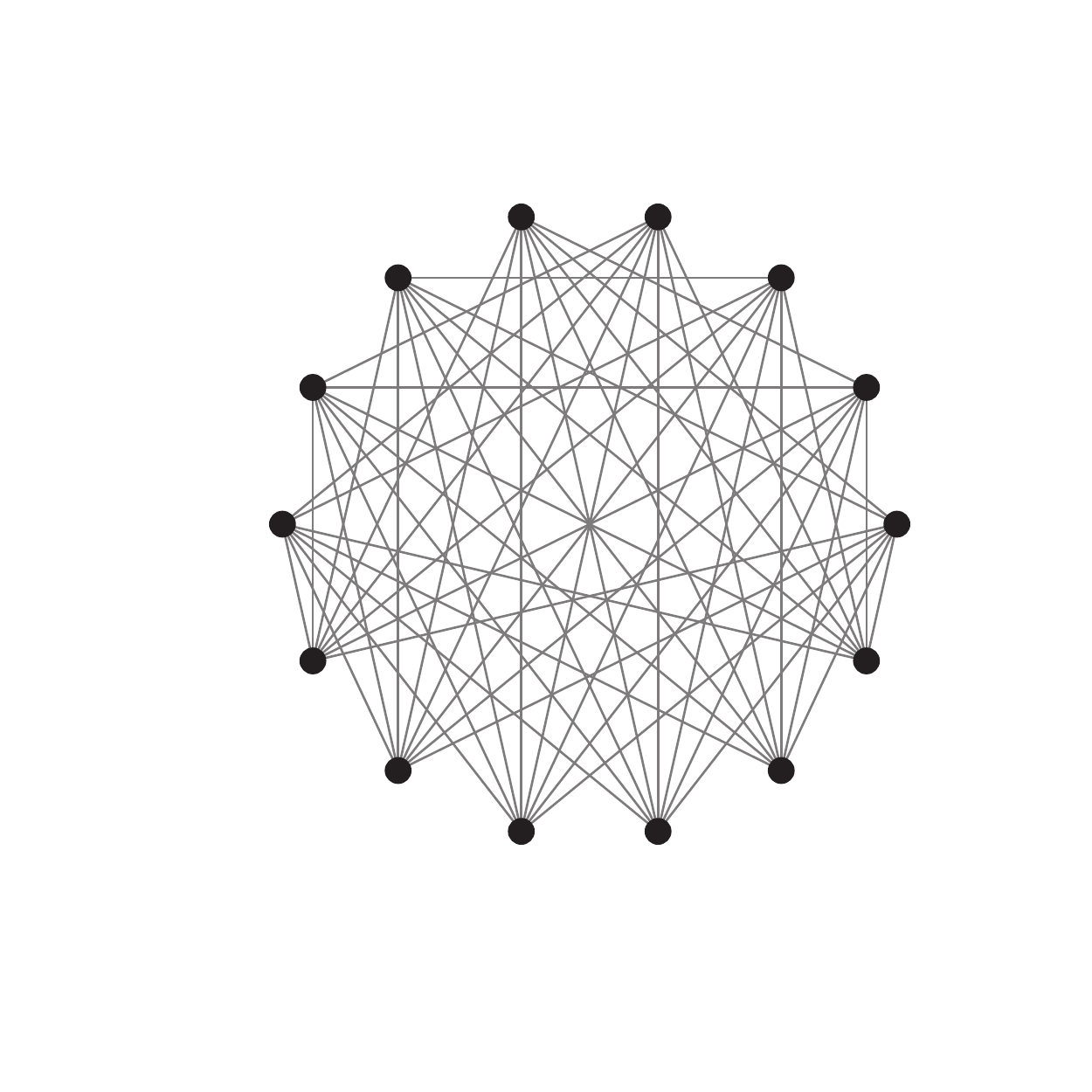}
		\caption{$n=14$, $e=56$\\ $N_t\approx 8.06\times10^{10}$}
		\label{n14k8}
	\end{subfigure}
		\hfill
		\begin{subfigure}[b]{0.3\textwidth}
		\centering
		\includegraphics[width=\textwidth]{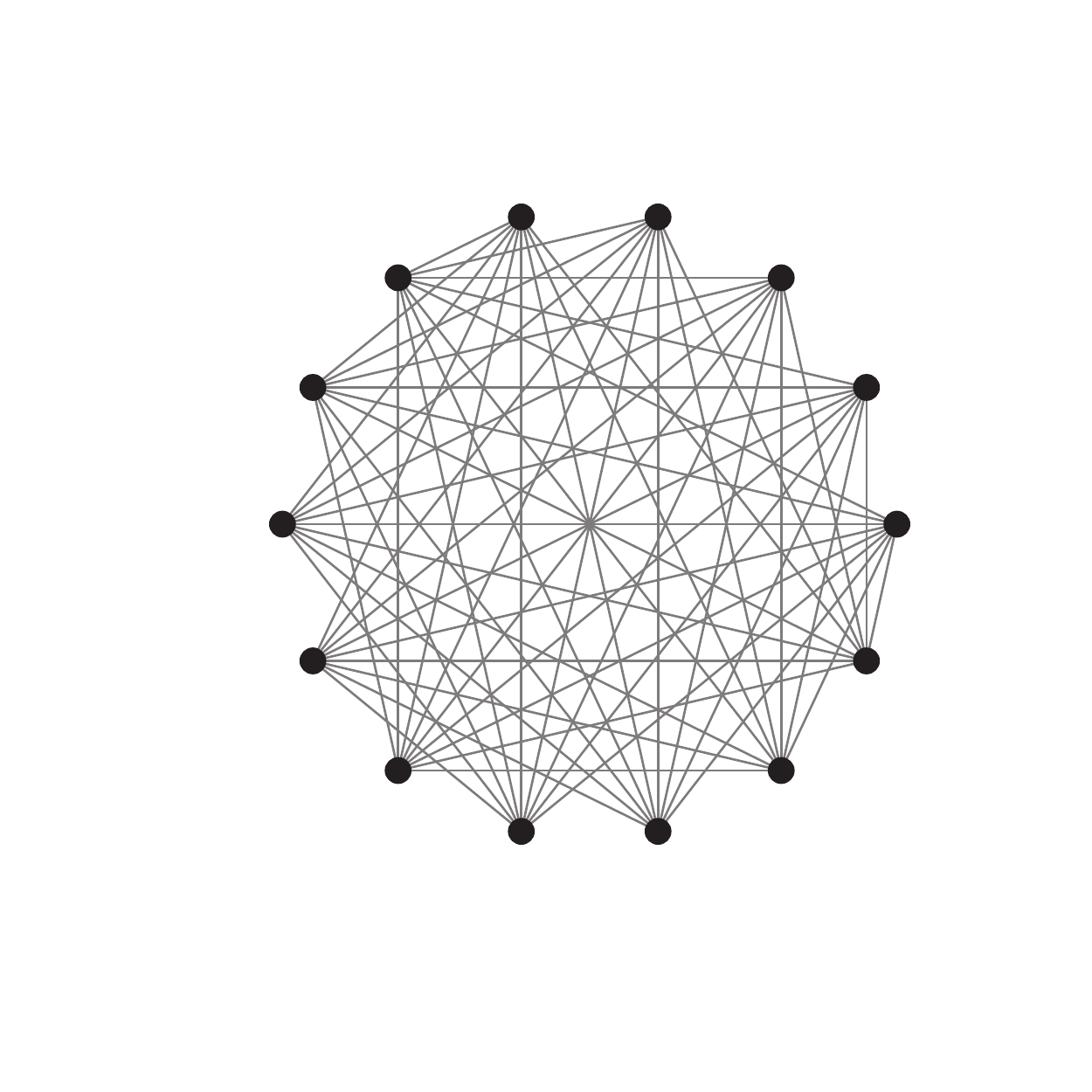}
		\caption{$n=14$, $e=63$\\ $N_t\approx8.06\times10^{10}$}
		\label{n14k9}
	\end{subfigure}
		\hfill
		\begin{subfigure}[b]{0.3\textwidth}
		\centering
		\includegraphics[width=\textwidth]{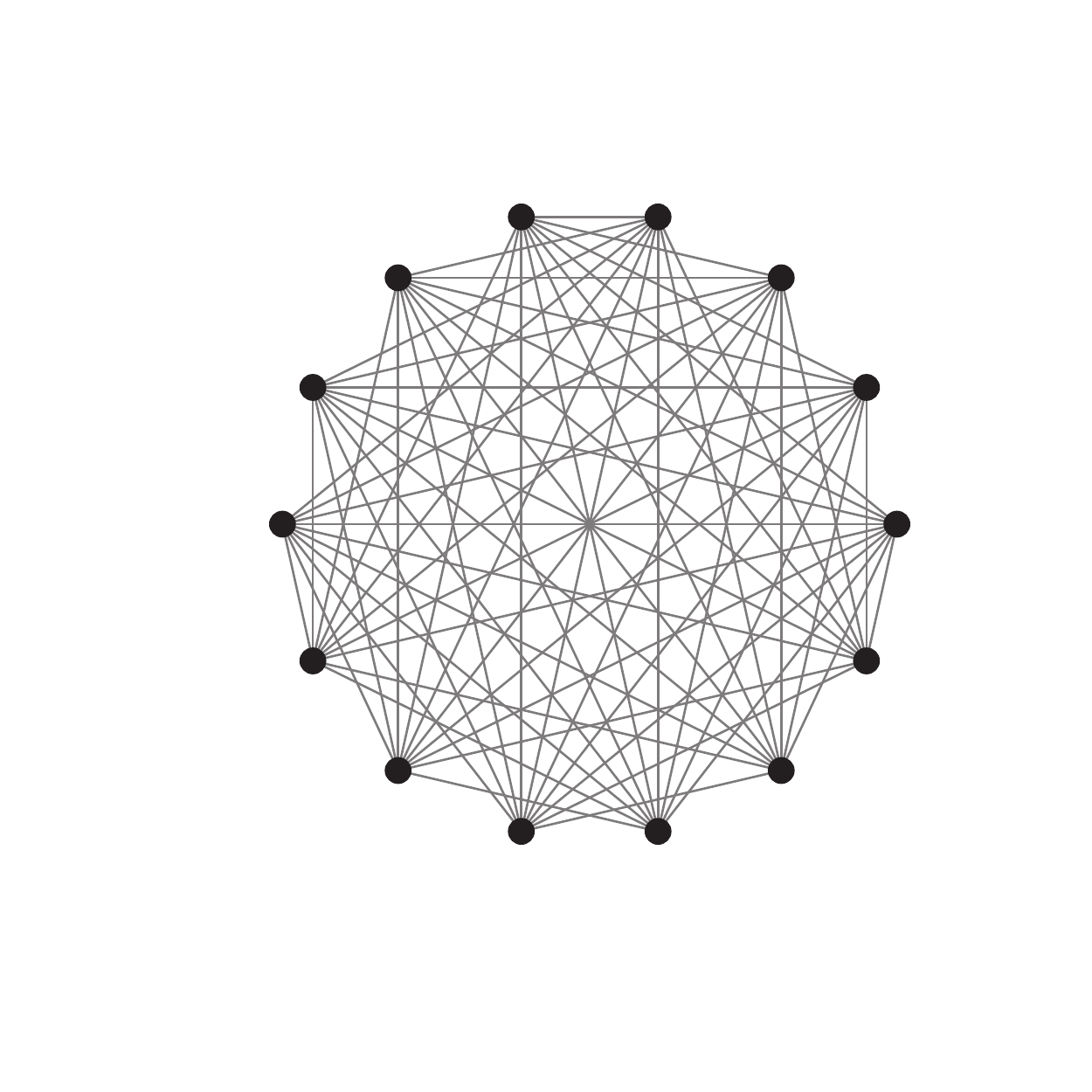}
		\caption{$n=14$, $e=70$\\ $N_t\approx1.65\times10^{12}$}
		\label{n14k10}
	\end{subfigure}
		\hfill
		\begin{subfigure}[b]{0.3\textwidth}
		\centering
		\includegraphics[width=\textwidth]{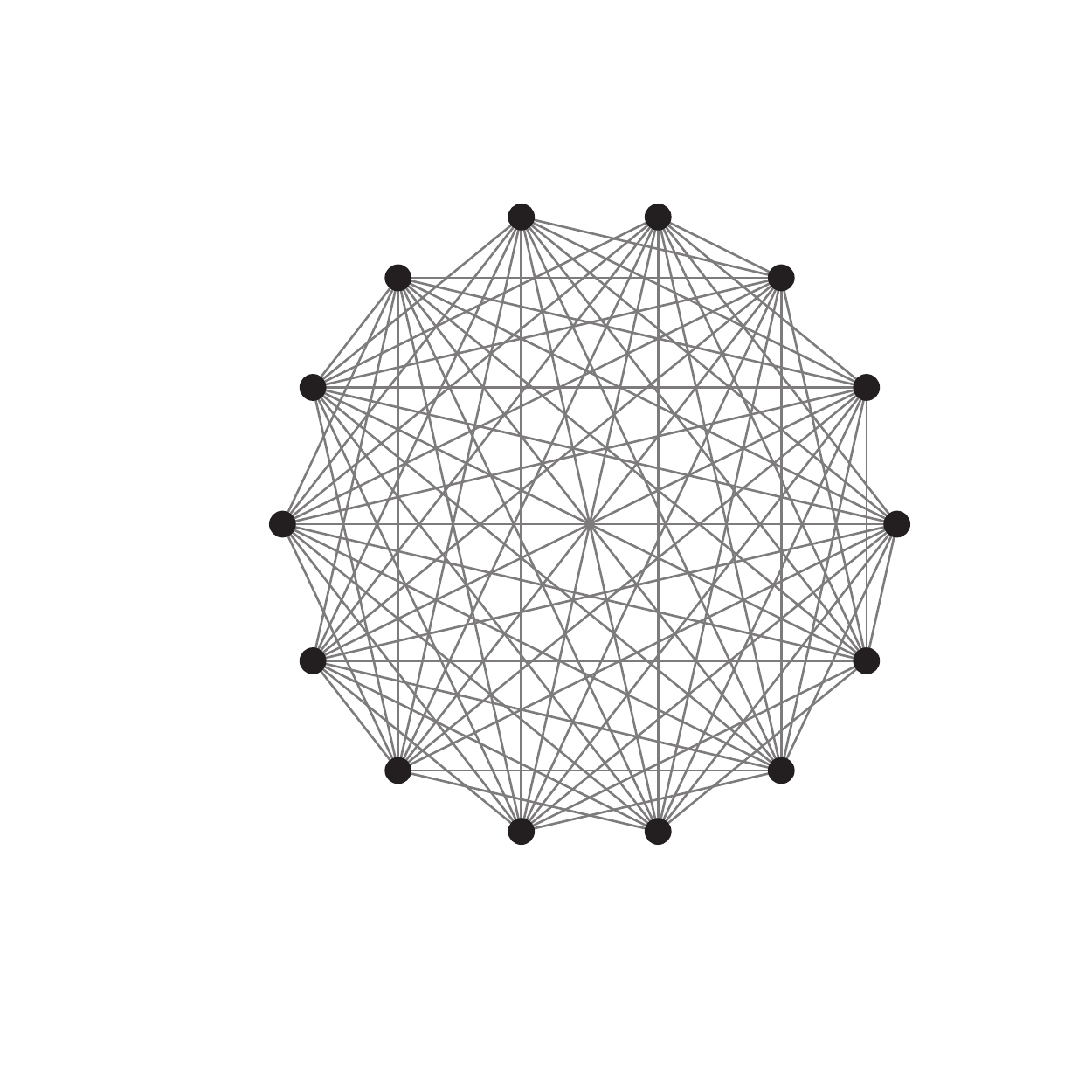}
		\caption{$n=14$, $e=77$ (graph $\overline{3C_3\cup C_5}$)\\ $N_t\approx5.96\times10^{12}$}
		\label{n14k11}
	\end{subfigure}
	\caption{Potential UMRG with $n=14$}
	\label{umrgn14}
\end{figure}

Every graph explored at this stage fulfilled the same conditions as smaller-order graphs. All of them had maximum girth, with no exceptions find among these classes. Most of them exhibit minimum diameter, except $(12,30)$ and  $(14,42)$-UMRG which has maximum diameter. No energy pattern was found among these graphs either. Energy values ranged from 0\% to 93\% as well as the Laplacian energy. Finally, algebraic connectivity (Fiedler number) was maximum in every graph except in $(10,20)$, $(12,18)$, $(12,42)$,$(14,28)$,$(14,35)$ and $(14,56)$ whose values corresponded to percentiles 98\% in the first two cases and 99\% in the last classes.

We can notice that for $e = n(n-3)/2$ the candidate  is always the complement of some copies of $C_3$ and a single copy of $C_{3+r}$ with $r = n \mod 3$. 
Let us express this observation as a conjecture.
\begin{conjecture}
For any $n\geq 5$ with $n= 3k+r$, $ 0 \leq r < 3$ and $e = n(n-3)/2$, then 
$\overline{(k-1)C_3\cup C_{3+r}}$ is UMRG.  
\end{conjecture}
By the results found by far, the conjecture has been verified for $n\leq 11$ and we have strong evidence for $n=12,13,14$, unless Question~\ref{question} has a negative answer.

One might wonder about that cycle lengths distribution among the possible complements of cycle unions. A possible answer, we found suggestive, is that this distribution minimizes the number of 3-cycle of a regular $(n, n(n-3)/2)$ graph. Indeed, it can be checked that \emph{the number of 3-cycles of the complement of $C_{h_1}\cup \dots \cup C_{h_k}$ with $\sum h_i = n$ is $N_{\{h_i\}_i}=\binom{n}{3}-\sum_i f_n(h_i)$ with 
$f_n(h) = h(n-3) + 1_{h=3}$.}
Therefore, 
$$
\sum_i f_n(h_i) = 
\sum_i h_i(n-3) + 1_{h_i=3} =
(n-3)\sum_i h_i + \sum_i1_{h_i=3} =
(n-3)n + |\{i: h_i = 3\}|,
$$
and $N_{\{h_i\}}$  is minimized when $|\{i: h_i = 3\}|$ is as large as possible, i.e., when   $k=\lceil n/3\rceil $ and $h_1 = \dots=h_{k-1}=3$.
These remarks, together with the one given after Proposition~\ref{Prop:either}, call us to reformulate the original Boesch's over the regular graphs to the following one.
\begin{conjecture}
The $(n,kn/2)$ UMRG, if exists, has maximum girth $g$ among the $k$-regular $(n,kn/2)$-graphs and  minimum number of $g$-cycles among the $k$-regular $(n,kn/2)$-graphs with girth $g$.
\end{conjecture}

\section{Concluding Remarks}
Our work suggest that there is still a long way to go regarding knowledge about network reliability. Collectively, our results are widely consistent with previous results about UMRG topologies. This study has developed two exact algorithms to find SUMRG for several orders and sizes. Although some of these graphs had been already described in literature, novel UMRG and candidates to UMRG are presented in this work. Notwithstanding the required large computational effort , this exploratory work provides a valuable insight for the study of most-reliable graphs. To the best of our knowledge, this is the first work to describe a counterexample to the long-held Boesch's conjecture that states that UMRG have the largest girth among all $(n,e)$-graphs: 4-regular $(9,18)$-UMRG has  girth $3$. In fact, this study also identified that the family of $k$-regular graphs with $2k+1$ nodes and ${k(2k+1)}/{2}$ edges, for even $k$, has girth 3 and the consequent proof is given. 
The second major finding of this research concerns the topology of candidates to UMRG with $n$ nodes and ${n(n-3)}/{2}$ edges. A new conjecture is posed, though not proved, in this article. Surprisingly, despite that it was expected to find a more homogeneous topology, being its complement a union of cycles of similar orders, a different result emerged from out research. For graphs with given size and order, the candidate to UMRG is $\overline{(k-1)C_3\cup C_{3+r}}$ where $n=3k+r$ with $0\leq r <3 $. Proving its optimality is a challenging task as it can be concluded from~\cite{grafok44}.  

A lot of questions remains to be answered and many of the conjectures that exist in the field of Uniformly Most Reliable Graphs are still open. Further research in this field would be useful to answer most of this questions and shed light on a topic that is still under-explored.

\section*{\normalsize{ACKNOWLEDGMENTS}}
This work is supported by Project 563 CSIC Proyecto de Iniciación a la Investigación: ``B\'usqueda de grafos uniformemente confiables y an\'alisis de sus propiedades''. This work is a product of the project funded by CSIC, proposed and mainly done by the first author under the direction of the second one.


\bibliography{biblio}

\clearpage

\appendix
\setcounter{table}{0}
\renewcommand{\thetable}{A.\arabic{table}}
\section{Number of 2-connected graphs}\label{apen1}
\begin{table}
	\caption{Number of 2-connected graphs per order and size }
	\begin{threeparttable}
		\begin{tabular}{lcc}
			\headrow
			\thead{Order} & \thead{Size} &  \thead{Number} \\
			\hiderowcolors
			6 & 9 & 14\\
			7 & 14 & 59\\
			8 & 12& 429\\
			8 & 16 & 1114\\
			8 & 20 &215\\
			9 & 18 &27015 \\
			9 & 27 & 765\\
				10 & 15 & 23370\\
			10 & 20 & 774876 \\
			10 & 25 & 1012187\\
			10 & 30 &135571 \\
			10 & 35 & 2763\\
			11 & 22 & 264\\
			11 & 33 & 266\\
			11 & 44 & 6 \\
			12 & 18 & 81\\
			12 & 24 & 1542 \\
			12 & 30 & 7848\\
			12 & 36 & 7849\\
			12 & 42 & 1547\\
			12 & 48 & 94\\
			12 & 54 & 9\\
			13 & 26 & 10768\\
			13 & 39 & 367860\\
			13 & 52 & 10786 \\
			13 & 65 & 10 \\
			14 & 21 & 480 \\
			14 & 28 & 88126 \\
			14 & 35 & 3459379 \\
			14 & 42 &  21609300\\
			14 & 49 & 21609301 \\
			14 & 56 &  3459386\\
			14 & 63 & 88193 \\
			14 & 70 & 540 \\
			14 & 77 & 13\\ 
			\hline  
		\end{tabular}
	\end{threeparttable}
	\label{tablagrafos1}
\end{table}

\clearpage

\section{Coefficients of the reliability polynomial}\label{apen2}
\begin{table}[h!]
	\caption{UMRG Coefficients $N_i$ ($N_5$-$N_{16}$)}
	\begin{threeparttable}
		\setlength\tabcolsep{1.5pt}
		\begin{tabular}{lccccccccccccc}
			\headrow
			\thead{n} & \thead{e} & \thead{$N_5$}& \thead{$N_6$}& \thead{$N_7$}& \thead{$N_8$}& \thead{$N_9$}& \thead{$N_{10}$} & \thead{$N_{11}$} & \thead{$N_{12}$}& \thead{$N_{13}$}& \thead{$N_{14}$}& \thead{$N_{15}$}& \thead{$N_{16}$}\\
			\hiderowcolors
			6 & 9 & 81 & 78 & 36 & 9& 1& NA& NA& NA& NA& NA& NA& NA \\
			7 & 14 & 0 & 1200 &	2460 & 2668 & 	1932 &	 994 &	364	& 91 & 14 &	1 & NA& NA \\
			8 & 12 & 0 & 0 & 392 &	409&	212&	66	&12&	1& NA& NA& NA& NA	\\
			8 & 16 & 0 & 0 & 4096&	8424&	9552&	7464&	4272&	1812&	560&	120&	16&	1 	\\
			8 & 20 & 0 & 0 & 21025	&69050&	124880&	159680&	156825&	122310&	76680&	38640&	15496&	4845	\\
			9 & 18 & 0 & 0 & 0 & 12480&	27856&	33772&	28344&	17725&	8442&	3051&	816	&153 \\
			9 & 27 & 0 & 0 & 0 & 419904&	1957032&	5128272&	9743436	&14661909&	18191007&	18999630&	16892658&	12854349
			\\
			\hline  
		\end{tabular}
	\end{threeparttable}
	\begin{tablenotes}
		\item	NA: Not Applicable for given $n$ and $e$
	\end{tablenotes}	
	\label{tablacoef}
\end{table}

\begin{table}[h!]
	\caption{UMRG Coefficients $N_i$ ($N_{17}$-$N_{27}$)}
	\begin{threeparttable}
		\setlength\tabcolsep{1.5pt}
		\begin{tabular}{lcccccccccccc}
			\headrow
			\thead{n} & \thead{e}& \thead{$N_{17}$}& \thead{$N_{18}$}& \thead{$N_{19}$}& \thead{$N_{20}$}& \thead{$N_{21}$}& \thead{$N_{22}$}& \thead{$N_{23}$}& \thead{$N_{24}$}& \thead{$N_{25}$}& \thead{$N_{26}$}& \thead{$N_{27}$}\\
			\hiderowcolors
			
			6 & 9 & NA& NA& NA& NA& NA& NA& NA& NA& NA& NA & NA \\
			6 & 12 & NA& NA& NA& NA& NA& NA& NA& NA& NA& NA& NA \\
			7 & 14 & NA& NA& NA& NA& NA& NA& NA& NA& NA& NA& NA \\
			8 & 12 & NA& NA& NA& NA& NA& NA& NA& NA& NA& NA& NA	\\
			8 & 16 & NA& NA& NA& NA& NA& NA& NA& NA& NA& NA& NA	\\
			8 & 20 & 1140&	190&	20&	1& NA& NA& NA& NA& NA& NA& NA	\\
			9 & 18 & 18&	1& NA & NA & NA& NA& NA& NA& NA& NA& NA \\
			9 & 27 & 8382393	&4674855&	2218185&	887841&	296001&	80730&	17550&	2925&	351	&27&	1
			\\
			
			\hline  
		\end{tabular}
	\end{threeparttable}
	\begin{tablenotes}
		\item	NA: Not Applicable for given $n$ and $e$
	\end{tablenotes}	
	\label{tablacoef2}
\end{table}

\end{document}